\documentclass{article}

\usepackage{graphicx,hyperref}
\usepackage{booktabs}
\usepackage{geometry}
\usepackage{caption}
\usepackage[utf8]{inputenc}\usepackage{amsmath,amssymb}
\usepackage{pgfplots}
\pgfplotsset{compat=1.18}
\usepgfplotslibrary{groupplots}
\usepackage{hyperref}
\usepackage[table]{xcolor}
\usepackage{array}
\usepackage{makecell}
\usepackage{adjustbox}
\usepackage{pdflscape}
\usepackage[font=small,labelfont=bf]{caption}

\usepackage[T1]{fontenc}
\usepackage{geometry}
\usepackage{float}
\usepackage{tikz,longtable}\usepackage{makecell}
\usepackage{tabularx}
\usepackage{booktabs} 

\usepackage{calc,pifont,eurosym,amsmath,wasysym,amssymb,amsfonts}
\usepackage[LGR,T1]{fontenc}
\usepackage[greek,english]{babel}
\usepackage{xcolor,microtype,longfbox}
\usepackage[explicit]{titlesec}
\usepackage[skip=0.35cm plus 0.07cm,indent=0cm]{parskip}
\usepackage{fancyhdr}
\usepackage{lmodern, multirow} 
\usepackage{amsthm}

\newtheorem{theorem}{Theorem}[section]   
\newtheorem{lemma}[theorem]{Lemma}       
\newtheorem{proposition}[theorem]{Proposition}

\usepackage{array,xcolor,colortbl,booktabs}

\definecolor{P4}{HTML}{DCEBF7}    
\definecolor{P8}{HTML}{B7D7F7}    
\definecolor{P16}{HTML}{76B9F2}   
\definecolor{P32}{HTML}{1E90FF}   
\definecolor{P64}{HTML}{0047AB}   

\definecolor{S8}{HTML}{FFDFA6}    
\definecolor{S16}{HTML}{FFB347}   
\definecolor{S32}{HTML}{FF8C00}   
\definecolor{S64}{HTML}{FF6600}   

\definecolor{SPEC}{HTML}{9370DB}  
\definecolor{INV}{HTML}{C71585}   

\newcolumntype{P}{>{\centering\arraybackslash}p{0.7cm}} 
\newcolumntype{Q}{>{\centering\arraybackslash}p{0.7cm}} 

\title{Stabilization and Regaining Periodicity in Modular Laplacian Dynamics}
\author{Małgorzata Nowak-Kępczyk}
\date{}

\begin{document}
	\maketitle

\begin{abstract}
We study discrete Laplacians on two–dimensional lattices under modular iterations, focusing on the emergence of nontrivial large–scale patterns.
While purely binary or constant modular sequences quickly collapse into strict periodicity, the insertion of a single non–binary step $k$ yields qualitatively new behavior.
Through extensive computer–assisted exploration we identify a taxonomy of long–lived figures — rugs, quasi–carpets, and carpets — whose occurrence depends systematically on seed symmetry, neighborhood mask, and sequence structure.
In particular, we show that mixed families of the form $[2,k,2^s]$ can stabilize high–density carpets beyond the universal decay time characteristic of binary dynamics.

Our approach combines algebraic replication laws with large–scale simulations and density tracking, producing both theoretical conditions (periodicity via Lucas’ theorem, non–overlap criteria) and experimental evidence of persistent quasi–aperiodic architectures.
The results highlight how minimal modifications in discrete local rules generate unexpectedly rich multiscale geometries, bridging rigorous analysis with computer–assisted discovery.
\end{abstract}


\section*{Introduction}

Discrete Laplacians on infinite lattices form a fundamental class of operators 
in discrete analysis and computational mathematics. Their iterations on simple 
seeds generate evolving patterns that can be studied both algebraically and 
geometrically, often yielding unexpectedly rich behavior \cite{Ilachinski2001}. 
From the dynamical systems perspective they resemble cellular automata and 
lattice gases, where local rules produce complex global structures.  

A striking feature of these systems is their ability to create intricate forms—
from Sierpiński-like gaskets to carpet-like tilings—despite their minimalistic 
definition. This places them at the intersection of experimental mathematics 
and computational modeling, extending the tradition of computer-assisted 
discoveries in dynamical systems, including fractals, percolation clusters, 
and quasi-crystals.

\paragraph{Related work.}
Periodicities and replication properties of discrete Laplacians were analyzed 
in earlier studies \cite{Hadlich2011}.  
Beyond mathematics, similar principles appear in self–assembly phenomena: 
Whitesides and Grzybowski \cite{Whitesides2002} highlighted the ubiquity of 
self–organization across scales, while Singh \emph{et al.} 
\cite{Singh2024NonEquilibrium} emphasized its role in sustaining 
living–matter–like architectures.  
Recent experiments even demonstrated the emergence of Sierpiński–triangle 
motifs in protein assemblies \cite{Sendker2024Fractal}, 
suggesting that discrete Laplacian dynamics may capture essential aspects 
of replication and stabilization in natural systems.

\paragraph{Previous studies.}
Our earlier works examined binary, ternary, and higher–order modular Laplacian 
dynamics, including their geometrical aspects, fractal and chaotic behavior, and 
algebraic interpretations \cite{Suzuki2018Geometrical,Lawrynowicz2019,NowakKepczyk2025Preprint}. 
These studies provide the broader background for the present work, where we focus 
on carpets and quasi–carpets generated by mixed modular sequences of the form $[2,k,2^s]$.

In this paper we explore periodicity, quasi-periodicity, and breakdown phenomena arising from modular extensions of discrete Laplacians on two-dimensional lattices. Our contributions are threefold:

We establish replication and non-overlap conditions for constant-k sequences, deriving clear k-adic periodic laws in the prime case.

We analyze alternating and perturbed binary sequences, showing how minimal modifications destroy strict periodicity and give rise to quasi-periodic and chaotic regimes.In this paper we explore periodicity, quasi-periodicity, and breakdown phenomena 
arising from modular extensions of discrete Laplacians on two-dimensional lattices. 
Our contributions are threefold:
\begin{itemize}
	\item We establish replication and non-overlap conditions for constant-$k$ sequences, 
	deriving clear $k$-adic periodic laws in the prime case.
	\item We analyze alternating and perturbed binary sequences, showing how minimal 
	modifications destroy strict periodicity and give rise to quasi-periodic and chaotic regimes.
	\item We introduce a taxonomy of emergent structures---rugs, quasi-carpets, and carpets---based 
	on seed/mask symmetries and density evolution, and identify the $[2,k,2^s]$ family 
	as a minimal setting capable of sustaining long-lived carpets.
\end{itemize}

Taken together, these results provide a systematic classification of discrete Laplacian dynamics. They bridge rigorous replication laws with computational taxonomy, and point toward interdisciplinary relevance wherever discrete local interactions generate complex spatial architectures.

	\section{Model and definitions}  
	
	In this work, we study sequences of \textit{discrete Laplacians} defined on an infinite, regular square lattice. The initial state of the system, denoted by $\nu_0$, is specified by a chosen \textit{seed}---a finite set of cells with value~1 (or other positive integers)---surrounded by an empty background (value~0).  
	
	\paragraph{Cellular automaton.}  
	We investigate the evolution of a cellular automaton whose local dynamics are governed by the discrete Laplacian operator:  
	\begin{equation}
		\Delta u(p) \,= \, \sum_{g \in N(p)} \big( u(g) - u(p) \big) \pmod{k_i}, \quad k_i \in \{2,3,4, \dots\},
	\end{equation}
	where $N(p)$ denotes the neighborhood of site $p$ according to a prescribed mask, and $k_i$ at $i$-th step od evolution is the $i$-th element of a chosen sequence of positive integers:  
	\begin{equation}
		k_1, k_2, k_3,\ldots, \qquad k_i \in \mathbb{Z}^+,\;\; k_i \geq 2,\;\; i=1,2,3,\ldots. \label{sequence}
	\end{equation}  
	In most cases the binary sequence: 2,2,2,2... is considered.
	At each iteration we evolve
	\[
	u_{t+1} = L u_t \pmod{k_t},
	\]
	where $k_t$ is the $t$-th element of the sequence (\ref{sequence}). If $k_t=2$, $k_t=3$, $k_t=4$, and so on, we refer to the corresponding step (Laplacian) performed as a binary, ternary, or quaternary iteration (Laplacian), respectively.

\paragraph{Seeds and figures.}  
Classically, in Laplacian constructions, a seed was considered to be a nonempty set of nonzero cells contained in a $3\times 3$ square. 
In this paper we shall call it a \emph{small seed}. 
A seed enclosed in a square of $18\times 18$ cells will be called a \emph{medium seed}, and one enclosed in a square of $84\times 84$ cells a \emph{large seed}.  
A \emph{figure} $F$ is any finite set of nonzero cells. In particular, any figure can itself be regarded as a seed. 
These three sizes will serve as reference seeds throughout the experiments.  

\begin{figure}[h]
	\centering (a)
	\includegraphics[width=0.6\textwidth]{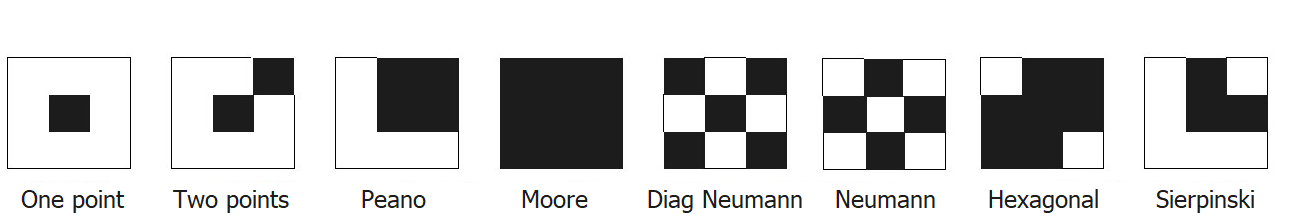} 
	$\;\;\;$(b)\includegraphics[width=0.07\textwidth]{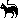}
	$\;\;$(c) \includegraphics[width=0.08\textwidth]{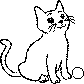}
	\caption{Examples of seeds: (a) small, (b) medium, and (c) large.}
	\label{seed}
\end{figure}  

\paragraph{Neighborhoods.}  
We employ a range of neighborhood masks of even and odd degree, including the standard Neumann, diagonal Neumann, Moore, Tannenbaum, as well as hexagonal and L-shaped configurations (see Fig.~\ref{neighbor}). 
In this paper, each sequence of Laplacians is assumed to use a constant chosen mask. 
For completeness, less frequently used masks (Tannenbaum, L-shaped) are also illustrated; their specific role will be clarified in the Results section.  

\begin{figure}[h]
	\centering
	\includegraphics[width=0.72\textwidth]{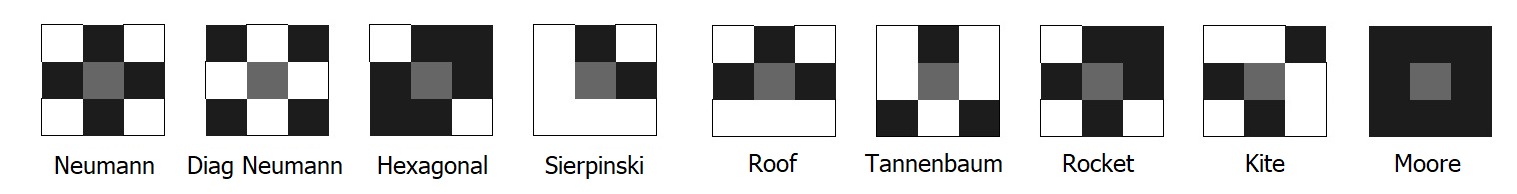}
	\caption{Examples of neighborhood masks.}
	\label{neighbor}
\end{figure}

\paragraph{Periods.}  
Let $L$ be a Laplace operator (depending on the neighborhood mask and modulus $k$ or $k_t$).  
Each figure $F$ can be tightly enclosed in a rectangle $r(F)$. 
A positive integer $\tau$ is called a \emph{period} of $F$ if  
\[
L^\tau(F)=\bigcup_{i=1}^s T_iF, \qquad s\ge 2,
\]  
for some lattice shifts $T_i$. Intuitively, after $\tau$ steps the figure decomposes into several shifted copies of itself.  

\begin{itemize}
	\item $\tau$ is a \emph{small period} if some enclosing rectangles $r(T_iF)$ overlap,  
	\item $\tau$ is a \emph{large period} if all $r(T_iF)$ are pairwise disjoint,  
	\item $\tau$ is a \emph{shifted period} if $F=L^{t_0}(S)$ for some seed $S$ and $t_0>0$, i.e.  
	\[
	L^{t_0+\tau}(S)=\bigcup_{i=1}^s T_i\,L^{t_0}(S).
	\]
\end{itemize}  

The \emph{minimal period} is the smallest $\tau$ satisfying these conditions.\footnote{On finite tori, every trajectory is ultimately periodic by Poincar\'e recurrence in a finite state space. In this paper we restrict attention to the infinite lattice $\mathbb{Z}^2$, where no such recurrence is guaranteed.}  

\paragraph{Relative density, period ratio, color densities, and entropy.}
Let $A_t=\{p\in\mathbb{Z}^2:\ u_t(p)\neq 0\}$ denote the support at time $t$, and let $r_t$ be the tight axis-aligned bounding rectangle of $A_t$. We define the
\emph{relative density} and the \emph{period ratio} by
\[
\rho_t\;=\;\frac{|A_t|}{|r_t|},\qquad 
\kappa_{t,\tau}\;=\;\frac{\rho_{t+\tau}}{\rho_t}.
\]
For moduli $k>2$ we also track \emph{per-residue densities}
\[
\rho_t^{(c)}\;=\;\frac{\big|\{p:\ u_t(p)\equiv c\ (\mathrm{mod}\ k)\}\big|}{|r_t|},\quad c=0,1,\ldots,k-1,
\]
and the \emph{color entropy}
\[
H_t\;=\;-\sum_{c=0}^{k-1}\rho_t^{(c)}\log \rho_t^{(c)}, \qquad (0\log 0:=0).
\]

Sharp seed returns yield pronounced minima of $\rho_t$ (thus $\kappa_{t,\tau}\ll1$ at hits), whereas ternary-rich nonperiodic sequences lack such regular dips or exhibit them only with a systematic shift (e.g.\ $t\equiv 2\pmod 8$ for $23222\ldots$).
	
\paragraph{Scope of the study.}
The aim of this paper is to investigate how the choice of seed size, 
neighborhood mask, and sequence of moduli affects the dynamics of discrete 
Laplacians. In particular, we examine replication laws, periodic and 
quasi-periodic behavior, and long-term compactness of figures. 
A central question is the identification of sequences of moduli that 
avoid the binary fate of periodic density loss and enable the persistence 
of quasi-stable carpet-like structures. Quantitative indicators such as 
density, period ratio, and entropy are employed to distinguish between 
periodic, quasi-periodic, and chaotic regimes.

\section{Results}

\subsection{Shape inheritance (global outlines from local masks)}
\label{subsec:shape-inheritance}

While neighborhood masks trivially constrain local adjacencies, it is \emph{not} a priori
clear that large-scale figures generated by Laplacian iterations must inherit the same
global symmetries. In many lattice models (e.g., cellular automata, percolation) emergent
shapes can differ substantially from local neighborhoods. Here we report a robust
\emph{shape inheritance} phenomenon for discrete Laplacians.

\paragraph{Bounding shapes.}
Let $A_t=\{p\in\mathbb{Z}^2: u_t(p)\neq 0\}$ be the support at time $t$ and
$\mathrm{Hull}(A_t)$ its convex hull. We quantify the global outline by either:
(i) the inertia tensor of $A_t$ (principal axes and anisotropy ratio), or
(ii) polygonal features of $\partial\mathrm{Hull}(A_t)$ (salient angle statistics).
These descriptors are invariant under translations and robust to sparse outliers.

\paragraph{Shape inheritance principle.}
A consistent observation across all experiments is that large–scale figures 
inherit the geometry of the neighborhood mask.  
For instance, Neumann and Moore masks produce square outlines, 
the diagonal Neumann mask yields diamonds, while kite and tannenbaum masks 
lead to triangular structures. The mapping mask $\to$ outline is summarized 
in Table~\ref{tab:shape_inheritance} and illustrated in Fig. \ref{fig:inheritance}.  

\begin{table}[h!]
	\centering
	\begin{tabular}{|p{4cm}|p{7cm}|}
		\hline
		\textbf{Neighborhood mask} & \textbf{Inherited large–scale shape} \\ \hline
		Moore, von Neumann & Square shapes (orthogonal symmetry) \\ \hline
		diag Neumann & Diamond shape (rotated squares) \\ \hline
		Kite & Skew triangular shapes \\ \hline
		Hexagonal & Hexagonal shapes \\ \hline
		Rocket & Pentagonal shape \\ \hline
		Tannenbaum, Roof & Triangular (Eiffel tower, inverted roof) shapes \\ \hline
	\end{tabular}
	\caption{Shape inheritance: neighborhood mask $\to$ global carpet outline.}
	\label{tab:shape_inheritance}
\end{table}

\begin{figure}[htbp]
	\centering
	\includegraphics[width=0.11	\textwidth]{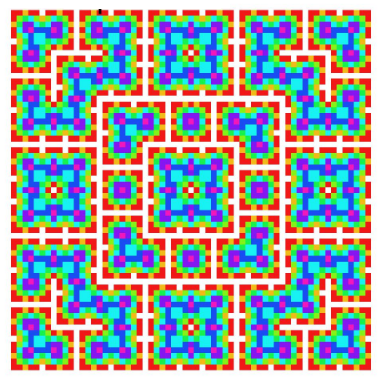}
	\includegraphics[width=0.11\textwidth]{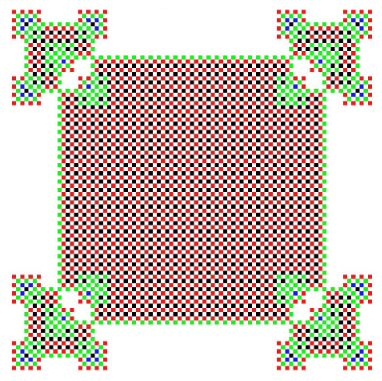}
	\includegraphics[width=0.11\textwidth]{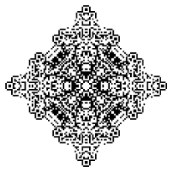}
	\includegraphics[width=0.11\textwidth]{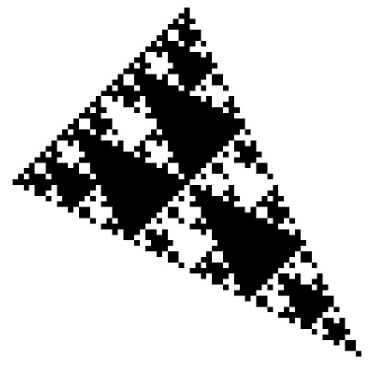}
	\includegraphics[width=0.11\textwidth]{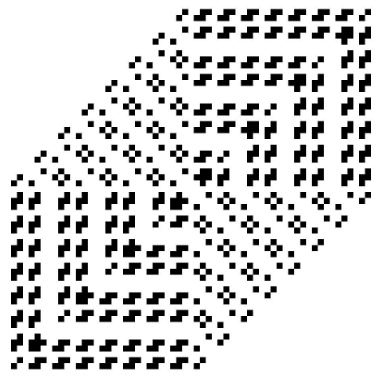}
	\includegraphics[width=0.11\textwidth]{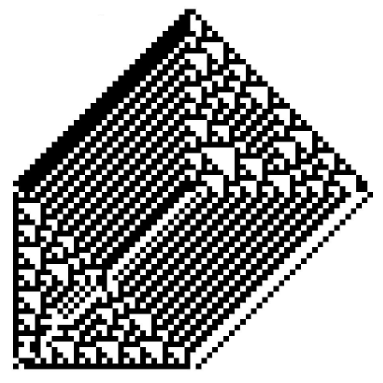}
	\includegraphics[width=0.11\textwidth]{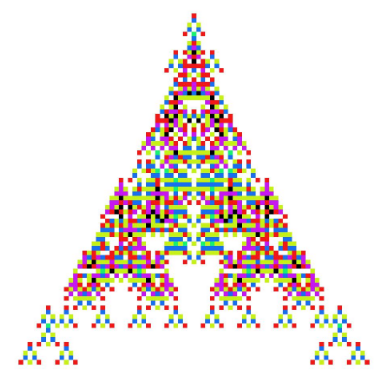}
	\caption{Shape inheritance across masks (medium seed; constant-$k$): 
		(a) Moore $\to$ square grid outline; 
		(b) diagonal Neumann $\to$ square outline; 
		(c) von Neumann $\to$ diamond outline;
		(d) kite $\to$ skew triangle outline;	
		(e) hexagonal $\to$ hexagon outline; 
		(f) rocket $\to$ pentagon-like outline.
		(g) tannenbaum $\to$ Eiffel tower-like outline.
		Outlines are extracted from $\partial\mathrm{Hull}(A_t)$; see Fig.~\ref{neighbor} for mask definitions.}
	\label{fig:inheritance}
\end{figure}

\paragraph{Remarks.}
(1) Shape inheritance principle is \emph{empirical but systematic}: it holds for all
tested combinations of seeds and moduli, and across long horizons $t$ where
replication or diffusion effects dominate local transients.
(2) The inheritance is measured at the level of the \emph{global envelope}
($\mathrm{Hull}(A_t)$ or inertia axes), not individual motifs; local microstructure
may break or refine symmetries without altering the outline class.
(3) In further subsections  we show that periodic
replication and its breakdown (e.g., in alternating or mixed sequences) modulate density
and entropy, yet the bounding shape remains predicted by the mask.

\paragraph{Quantitative outline descriptor.}
To support the visual classification, we measured simple geometric features of the 
bounding hull $\partial\mathrm{Hull}(A_t)$. In all cases, the observed outlines 
(square, diamond, hexagon, pentagon) aligned with the symmetry of the underlying 
mask. This confirms that the large-scale envelope of the figure inherits the 
mask geometry, even when local microstructure becomes irregular. 

\medskip
This shape inheritance provides a unifying constraint linking local update rules 
to global envelopes. In the next subsections we refine this picture by analyzing 
replication laws in constant-$k$ sequences and their breakdown under alternating 
or mixed schedules.

\subsection{Constant $k$-nary Laplacian sequence $kkkk\ldots$}

\paragraph{Frobenius endomorphism and periodicity of Laplacian automata}

Consider a linear cellular automaton over a finite field $\mathbb{F}_p$ 
with local update operator
\[
T = I + B,
\]
where $I$ is the identity and $B$ is a convolution operator given by the neighborhood mask.
Recall, that in a field of characteristic $p$, the \emph{Frobenius endomorphism} is
\[
\varphi(x) = x^p, \qquad x \in \mathbb{F}_p.
\]
Since $(a+b)^p = a^p + b^p$ in $\mathbb{F}_p$, more generally
\[
(a+b)^{p^m} = a^{p^m} + b^{p^m}.
\]
Consequently, applying this to $T = I+B$ gives
\[
T^{p^m} = (I+B)^{p^m} = I + B^{p^m}.
\]
All cross terms vanish modulo $p$, so after $p^m$ iterations the automaton 
reproduces the seed in the central window, with additional shifted copies 
generated by $B^{p^m}$ at the boundary.

\paragraph{Examples.}
\begin{itemize}
	\item For $p=2$, revivals occur at times $2^m$ ($1,2,4,8,\dots$).
	\item For $p=3$, at times $3^m$ ($1,3,9,27,\dots$).
	\item For $p=5$, at times $5^m$, and so on.
\end{itemize}

\paragraph{Summary.}
Thus the periodic revival of seeds is an immediate consequence of 
the Frobenius endomorphism in finite fields: raising $(I+B)$ to a 
$p^m$-th power eliminates mixed terms, leaving only the identity and 
a shifted convolution.

\subsection*{Binary sequence $222\ldots$}

\begin{proposition}
	Let $L$ be the discrete Laplacian (von~Neumann neighborhood) on $\mathbb{Z}^2$, acting modulo~2.  
	For any fixed residue $r\in\{0,\dots,7\}$, the sequence
	\[
	u_{r+8k} = L^{8k} u_r
	\]
	exhibits self-similar replication: each application of $L^8$ maps the current pattern into four disjoint translated copies at offsets $(\pm16,0)$ and $(0,\pm16)$.
\end{proposition}

\begin{proof}[Sketch of proof]  (The proof of Proposition 1 for the case of one‑point seed and the diagonal Neumann neighborhood was given in ~\cite{NowakKepczyk2025Preprint})
Write
\[
L = S_x + S_{-x} + S_y + S_{-y} - 4I,
\]
where $S_{\pm x},S_{\pm y}$ are unit shifts and $I$ is the identity.  
Over $\mathbb{F}_2$, we expand
\[
L^8 = \sum_{\alpha+\beta+\gamma+\delta+\epsilon=8}
\binom{8}{\alpha,\beta,\gamma,\delta,\epsilon}
S_x^\alpha S_{-x}^\beta S_y^\gamma S_{-y}^\delta I^\epsilon
\pmod{2}.
\]
By Lucas’ theorem, a multinomial coefficient is odd only when each summand index fits into the binary digits of $8_{(2)}=1000_2$. This forces $\alpha,\beta,\gamma,\delta,\epsilon\in\{0,8\}$ and kills all mixed terms.  
Thus the only surviving monomials are
\[
S_x^{\pm 16}, \qquad S_y^{\pm 16},
\]
and their combinations.  

Consequently $L^8$ acts as a sum of four pure translations, producing disjoint copies of the seed at offsets $(\pm16,0)$ and $(0,\pm16)$.  
Iterating the argument shows that $(L^8)^k$ yields a $2^k\times 2^k$ grid of replicas, with spacing doubled each time. This establishes the replication property. \end{proof}

\begin{figure}[h!]
	\centering
	(a)\includegraphics[width=0.75\textwidth]{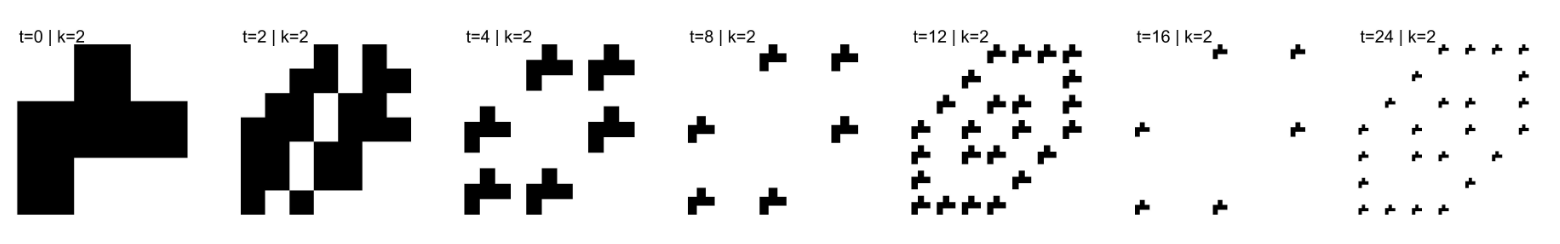}\\[1ex]
	(b)\includegraphics[width=0.75\textwidth]{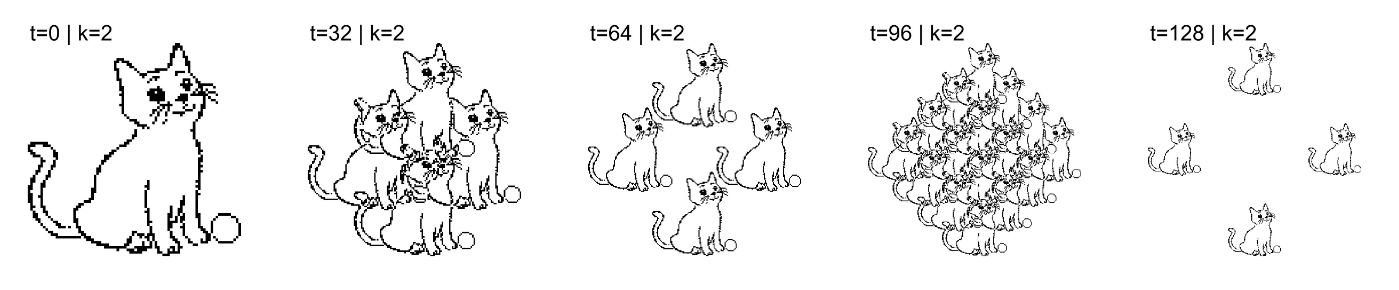}
	\caption{Illustration of small vs.\ big periods in constant binary sequence: 
		(a) small seed, hexagonal mask: rapid seed returns (big period T=2). 
		(b) big seed, Von Neumann mask: overlapping replications at 32 (small period); disjoint replication at $t=128$.}
	\label{fig:small-big}
\end{figure}

\paragraph{From revival to separation.}
While the Frobenius endomorphism explains why seeds reappear at
$p^m$ steps, it does not guarantee that the replicas are disjoint (see Fig. \ref{fig:small-big}).
In practice, early revivals overlap and only later iterations
yield separated copies, which we call the \emph{big period}.
Next we analyze this transition in detail.

\begin{lemma}[Big period via non-overlap]
	Let the seed have size $s\times s$. After $t$ iterations the occupied region 
	has bounding box $(s+2t)\times(s+2t)$. Suppose that at replication time 
	$t\in\mathcal{R}_k$ the pattern appears in a grid of $w_h$ copies 
	horizontally and $w_v$ copies vertically. Then the replicas are disjoint 
	whenever the spacing $2t$ between adjacent copies is at least as large as 
	the corresponding seed extension, i.e.
	\[
	2t \;\ge\; \max\{w_h,\,w_v\}\cdot s .
	\]
	In particular, for symmetric grids ($w_h=w_v=w$) this condition reduces to 
	\[
	T_{\mathrm{big}}(s;k)=\min\{\,t\in\mathcal{R}_k:\; 2t \ge ws \,\},
	\]
	where $\mathcal{R}_k$ is the (mask–dependent) set of replication times.
\end{lemma}

\paragraph{Remark.}
	In all prime-base cases considered ($k=2,3,5,7$), the replication grids are 
	symmetric ($w_h=w_v$), so the simplified condition $2t \ge w s$ applies.

\subsection*{Prime bases $k=3,5,7,\dots$}

For prime $k$, replication occurs on $k$-adic scales. By Lucas’ theorem in base~$k$, only pure translations survive in $L^{k^2}$, giving replication times $t=m\cdot k^2$.  
These copies become disjoint precisely when the non-overlap condition $2t \ge ws$ holds.

Similarly, for any other prime value of $k$, one obtains $k$-adic replication laws of the same form.

\paragraph{Prime vs.\ power-of-prime bases.}
For prime $k$ the $k$-adic replication law follows directly from
Lucas’ theorem. The same behavior extends to prime powers such as
$k=4,8,9$, where periodicity is governed by the underlying prime
base: $k=4$ and $k=8$ follow binary laws, while $k=9$ follows ternary
ones. In this sense powers of a prime behave ``as prime’’ with respect
to replication, and the distinction between small and big periods is
controlled by the same non-overlap condition.

\begin{figure}[h!]
	\centering 
(a) \includegraphics[width=0.82\textwidth]{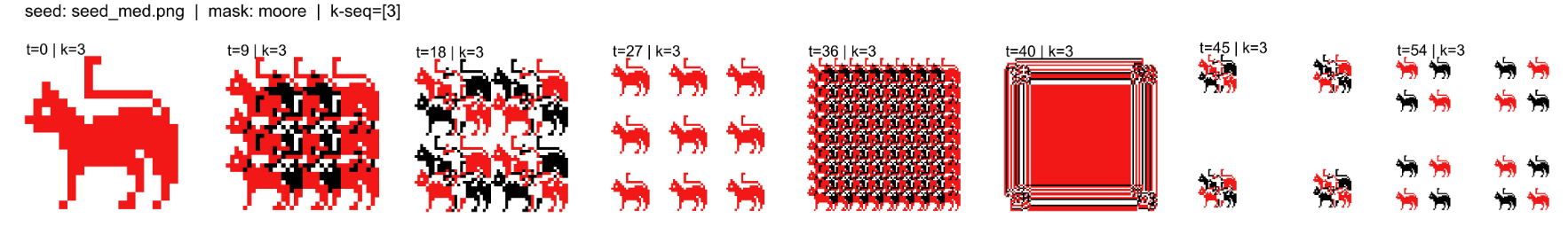}\\
(b) \includegraphics[width=0.82\textwidth]{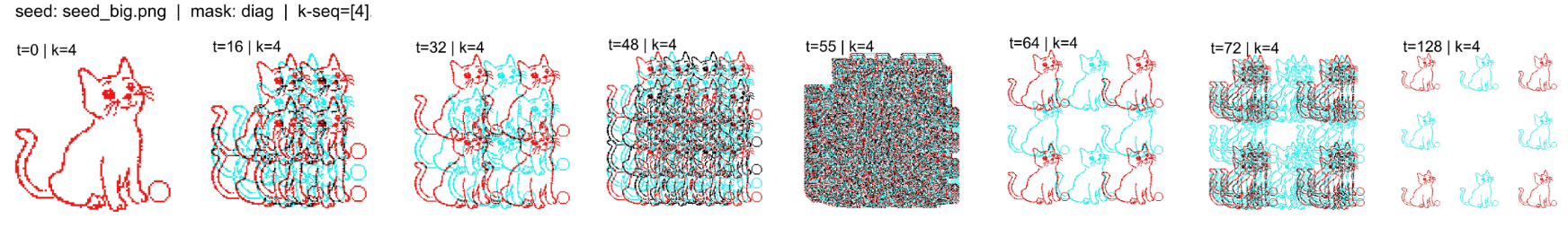}\\
(c) \includegraphics[width=0.82\textwidth]{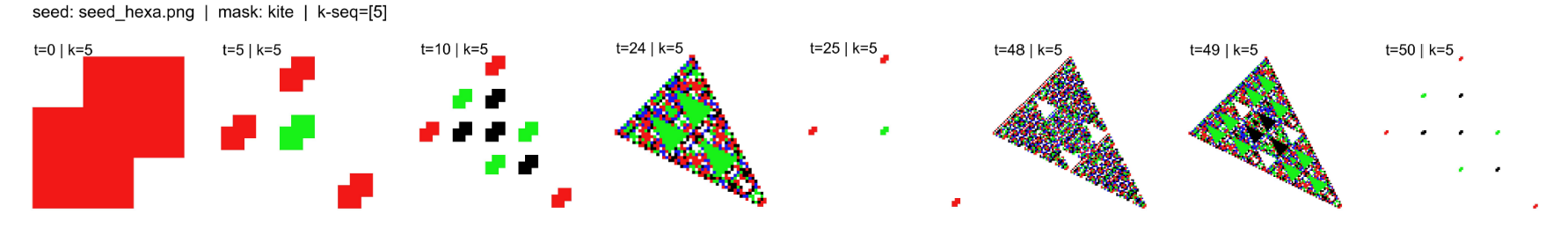}\\
(d) \includegraphics[width=0.82\textwidth]{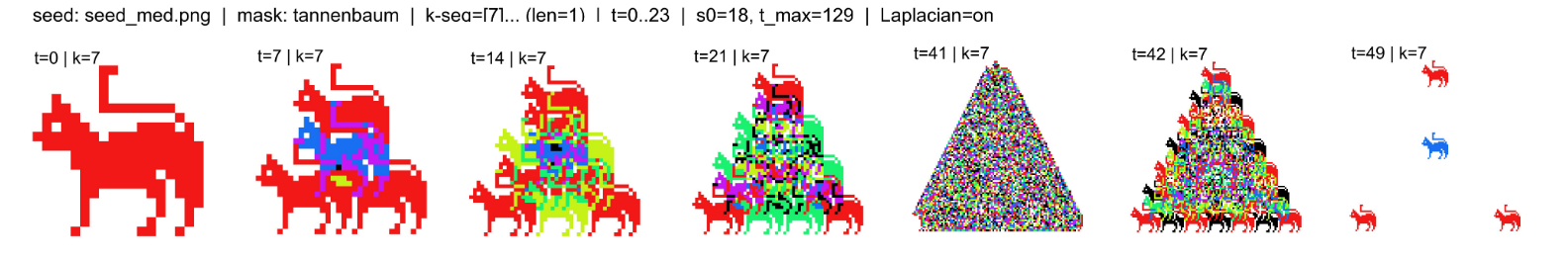}
	\caption{(a) Ternary sequence ($k=3$), medium seed, Moore mask. Early multiples ($9, 18,\dots$) overlap; higher multiples (of 27) yield disjoint seeds as $2t \ge ws$, (b) Quaternary sequence ($k=4$), large seed, diag-Neumann mask. Copies appear every 8-th step but disjoint ones at 128. (c) Quinary sequence ($k=5$), small seed, kite mask. Multiples of $5$ produce separate seeds: small and big periods coincide. (d) Septenary sequence ($k=7$). Multiples of $49$ generate replicated grids; disjoint copies appear when $2t \ge ws$. }
	\label{fig:ternary}
\end{figure}

\begin{table}[h!]
	\centering
	\caption{Small vs.\ big periods for constant-$k$ Laplacian sequences.
		Big period follows the non-overlap rule $2t\ge ws$ (see Lemma).}
	\label{tab:small-big-periods}
	\begin{tabularx}{\textwidth}{>{\raggedright\arraybackslash}p{2.6cm} p{5cm} X}
		\toprule
		\textbf{Sequence} & \textbf{Small period(s)} & \textbf{Big period (seed $s\times s$)} \\
		\midrule
$\mathbf{222\ldots}$ (binary)
& $\rule{0pt}{12pt}16,\,32,\,64,\,96,\dots$ (overlapping returns; self-similarity at scale $8$)
& $T_{\mathrm{big}}(s)=\min\{t\in\mathcal{R}_2:2t\ge ws\}$ \\
		
$\rule{0pt}{12pt}\mathbf{444\ldots}, \mathbf{888\ldots}$ (powers of two)
& Periods on powers-of-two scales (same as binary; mask dependent)
& $T_{\mathrm{big}}(s)=\min\{t\in\mathcal{R}_{2^m}:2t\ge ws\}$ \\

$\rule{0pt}{12pt}\mathbf{333\ldots}$ (ternary)
& Multiples of $27$: $27,\,54,\,81,\,108,\,135,\,162,\dots$
& $T_{\mathrm{big}}(s)=\min\{t\in\{27m:m\in\mathbb{N}\}:2t\ge ws\}$ \\

$\rule{0pt}{12pt}\mathbf{999\ldots}, \mathbf{27\,27\,27\ldots}$ (powers of three)
& Multiples of $81$: $81,\,162,\,243,\dots$ (same rule as ternary, base $3$)
& $T_{\mathrm{big}}(s)=\min\{t\in\{81m:m\in\mathbb{N}\}:2t\ge ws\}$ \\
		
		$\rule{0pt}{12pt}\mathbf{555\ldots}$ (quinary)
		& Multiples of $25$:$\rule{2.5cm}{0cm}$ $25,\,50,\,75,\,100,\,125,\,150,\dots$
		& $T_{\mathrm{big}}(s)=\min\{t\in\{25m:m\in\mathbb{N}\}:2t\ge ws\}$ \\
		
		$\rule{0pt}{12pt}\mathbf{666\ldots}$ (composite)
		& Mixed binary–ternary subperiods: binary-like ($32,64,96,\dots$), ternary-like ($27,54,81,\dots$)
		& Clean separation only when both scales align:
		$T_{\mathrm{big}}(s)=\min\{t\in\mathcal{R}_{2}\cap\mathcal{R}_{3}:2t\ge ws\}$ \\
		
		$\rule{0pt}{12pt}\mathbf{777\ldots}$ (septenary)
		& Multiples of $49$:$\rule{2.5cm}{0cm}$ $49,\,98,\,147,\,196,\dots$
		& $T_{\mathrm{big}}(s)=\min\{t\in\{49m:m\in\mathbb{N}\}:2t\ge ws\}$ \\
		\bottomrule
	\end{tabularx}
\end{table}

\subsection*{Composite base $k=6$}

Since $6=2\cdot 3$, dynamics intertwine binary and ternary scales.  
Small periods reflect binary-like ($32,64,96,\dots$) and ternary-like ($27,54,81,\dots$) returns.  
Clean separation requires alignment of both scales; the big period is given by the lemma with $t\in\mathcal{R}_2\cap\mathcal{R}_3$ (see Fig. \ref{fig:periods-6}).

\begin{figure}[h!]
	\centering
	\includegraphics[width=0.9\textwidth]{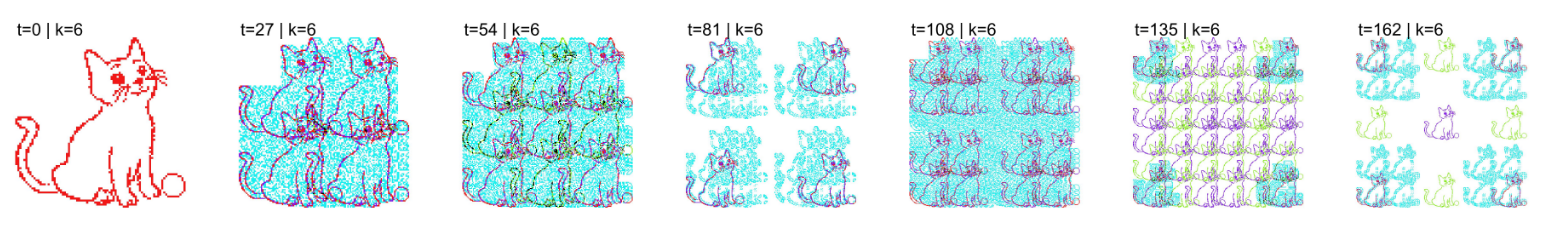}
	\caption{Composite sequence ($k=6$). Alternating binary-like and ternary-like phases. Clean separation is delayed until both scales align.}
	\label{fig:periods-6}
\end{figure}

\paragraph{Color entropies.}  
The behavior is particularly transparent in the entropy traces of 
individual colors (Fig.~\ref{fig:Hc-6}).  
For $k=3,4,5$ the colored entropies remain synchronized, 
reflecting the underlying strict periodicity.  
 
Interestingly, entropy fluctuations occur at times typical for binary and ternary periods: 8, 16, 27, 32, 54, 64, 81, 128. Thus, the desynchronization of colored entropies in $k=6$ directly reflects the interference between binary and ternary replication scales.

\begin{figure}[h!]
	\centering
	\includegraphics[width=0.9\textwidth]{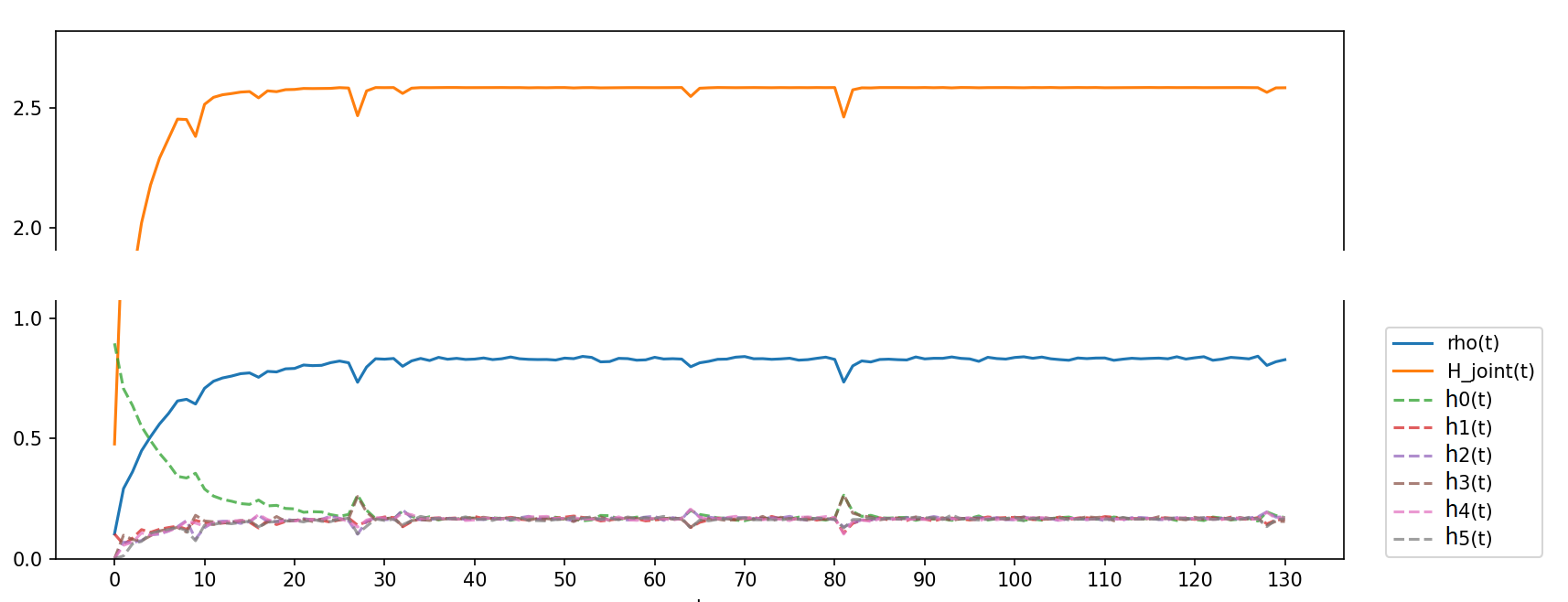}
	\caption{Color entropies in the constant sequence $6666\dots$ 
		with a large seed and Moore mask.  
		Note the appearance of entropy fluctuations at times typical for binary and ternary periods: 8, 16, 27, 32, 54, 64, 81, 128.}
	\label{fig:Hc-6}
\end{figure}

\paragraph{Mini-summary.} 
Prime moduli produce clean $k$-adic replication laws in constant $kkk\ldots$ sequences, with small and big
periods determined by Lucas’ theorem and the non–overlap lemma.  
Prime powers (e.g.\ $k=4,8,9$) follow the same rule as their underlying base, 
so $k=4,8$ are effectively binary while $k=9$ is ternary.  
Composite moduli (e.g.\ $k=6$) mix binary and ternary scales, delaying clean
replication. Periodicity is robust in the prime and prime-power cases, but fragile 
in the composite one.

Table~\ref{tab:small-big-periods} condenses the replication laws across different $k$.
 
The purely binary sequence $[2,2,2,\ldots]$ looks deceptively simple.  
Its behavior is completely predictable: densities oscillate with strict 
periodicity and figures replicate in a rigid grid, eventually thinning. 
The configuration reduces to a handful of dissipated seeds, from which 
a new cycle begins. Thus the binary case appears fully predictable, 
leaving little room for variation. This naturally raises the question 
whether a minimal perturbation of the sequence can alter the outcome. 
In particular, can one introduce a slight modification that destroys 
the trivial periodic scenario? This question motivates our exploration 
of the mixed families $[2,k,2^s]$.

Similarly, we observe analogous predictability in other constant sequences 
such as $[k,k,k,\ldots]$. Here, too, the dynamics ultimately collapse into strict 
periodicity, with the period given by a small power of $k$ (depending on the chosen 
neighborhood). Remarkably, even large and seemingly chaotic seeds, after a brief 
transient phase, suddenly crystallize into multiplied replicas on a rigid grid. 

An exception arises for composite values of $k$ built from two distinct primes 
(such as $k=6$ or $k=10$). In these cases we observe quasi-periodic fluctuations 
rather than pure repetition, yet without giving rise to qualitatively new 
structures. The dynamics remain constrained by the prime components of $k$, 
producing binary- and ternary-like subperiods whose interference prevents 
stable global replication.

\subsection{Alternating binary–$k$ sequences $[2k2k\ldots]$}

These sequences provide the minimal perturbation of the constant binary case, interleaving binary steps with modulus-$k$ steps.

		Consider alternating updates
	\[
	u_{t+1} =
	\begin{cases}
		L u_t \pmod{2}, & t \text{ even},\\
		L u_t \pmod{$k$}, & t \text{ odd}.
	\end{cases}
	\]
	
	\paragraph{Sequences 2$k$2$k$..., $k$ even.}
	
	The dichotomy between even and odd $k$ reflects the Chinese Remainder decomposition: for odd $k$ the system effectively evolves modulo $2k$, while for even $k$ binary periodicity partly survives.
	
	In case of big and medium seeds we obtain periodic figures with small period 16 and shift 1.
	The small seed case gives two possibilities:

	\begin{itemize}
		\item If the seed or mask is not double symmetric we obtain periodic figures.
		\item If the seed and mask are both double symmetric we obtain double symmetric, periodic figures of periodically dropping density.
	\end{itemize}	
	
	Double symmetry here means invariance under both horizontal–vertical and diagonal reflections.
	Illustration of the above observations can be seen in Figs. \ref{2k2k_big2} and \ref{2k2k_small2}.

	\begin{figure}[!]
		\centering
		(a)\includegraphics[width=0.82\textwidth]{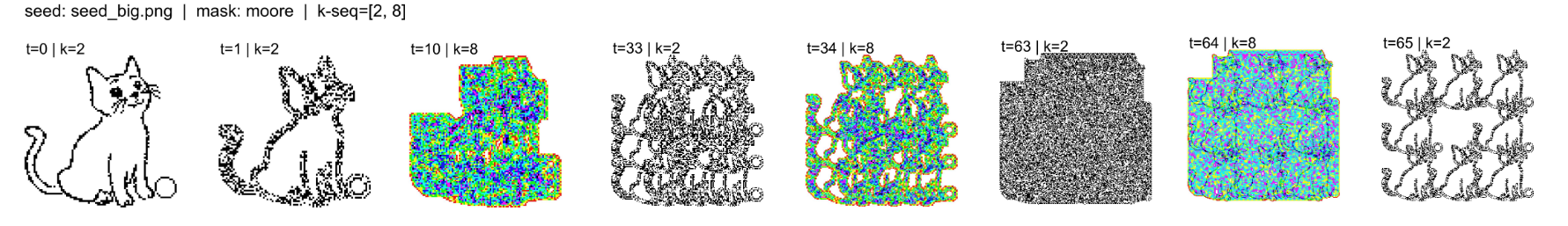}\\
		(b)\includegraphics[width=0.82\textwidth]{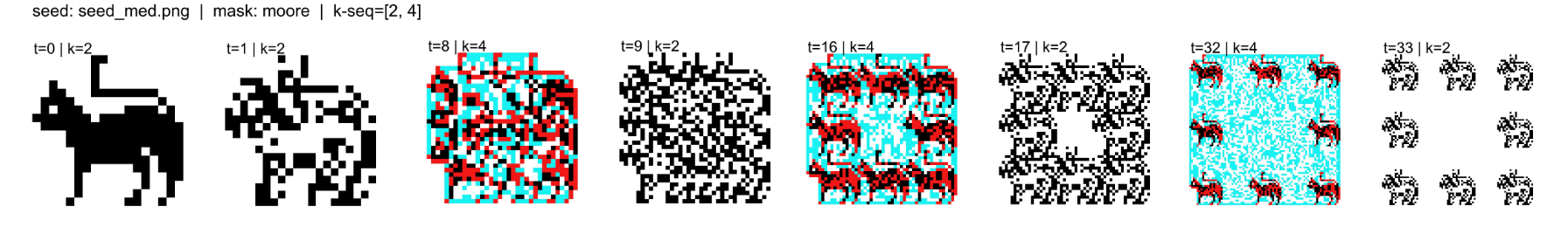}
		\caption{Big and medium seeds under alternating one-to-one 2$k$2$k$... iterations, $k>2$ even, (a) big and (b) medium seeds on Moore mask, $k=7, 3$, resp. resulting in periodic 1-shifted figures of small period 16.}
		\label{2k2k_big2}
	\end{figure}

\begin{figure}[h!]
	\centering
	(a)\includegraphics[width=0.82\textwidth]{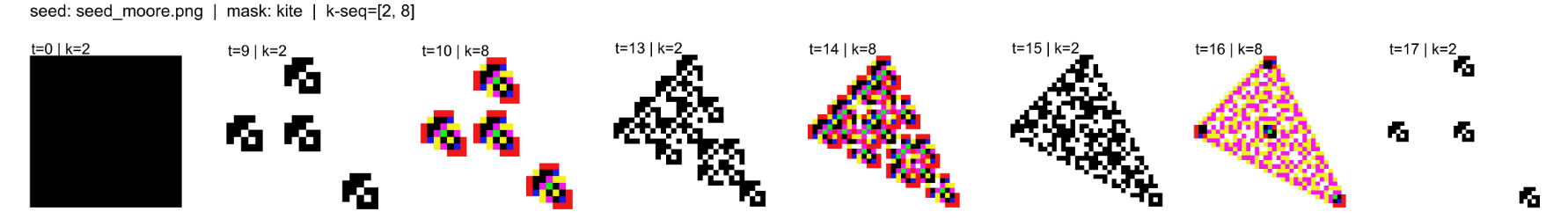}\\
	(a)\includegraphics[width=0.82\textwidth]{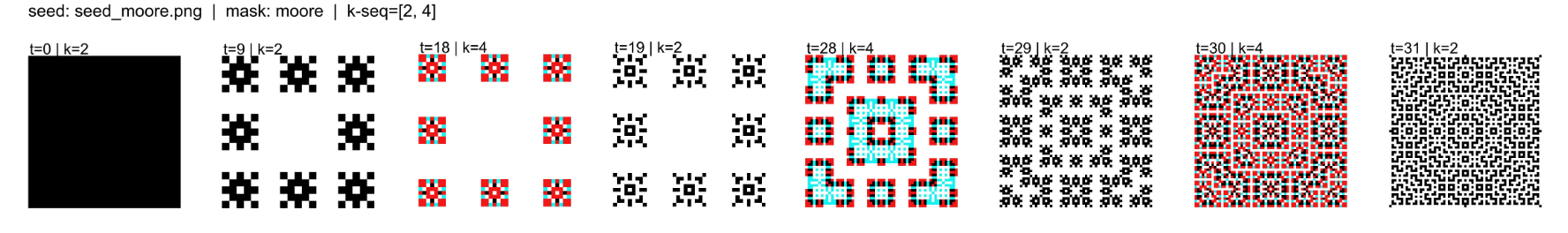}\\
	\caption{Small seeds under alternating one-to-one 2$k$2$k$... iterations, $k>2$ even, (a) non-symmetric seed and/or mask resulting in periodic non-double symmetric figures (here $k=7$, kite mask), (b) symmetric seed and mask resulting in periodic, double symmetric figures with periodically dropping density (diag Neumann mask).}
	\label{2k2k_small2}
\end{figure}

\paragraph{Sequences 2$k$2$k$..., $k>2$ odd.}
	The case of $k=3,5,7...$ is equivalent to working modulo $2k$, via the Chinese Remainder Theorem.
	The $2k2k\ldots$, $k$ odd, sequence destroys the short binary periods;
	empirically no return to seed is observed within practical bounds.
	Thus introducing ternary, quinary, etc.  steps breaks periodicity and increases complexity. 
	
	In case of big and medium seeds the figures obtained become quickly chaotic. The case of small seeds can be split in two:
	\begin{itemize}
		\item If the seed and mask are not symmetric or have conflicting axes of symmetry we obtain chaotic figures.
		\item If the seed and mask are both double symmetric we obtain double symmetric, non periodic figures of relatively high density ornamental carpet-like figures (more on shape inheritance in section \ref{inherit}).
	\end{itemize}
	
	Illustration of the above observations can be seen in Figs. \ref{2k2k_big} and \ref{2k2k_small}.
	
		\begin{figure}[h!]
		\centering
		(a)\includegraphics[width=0.82\textwidth]{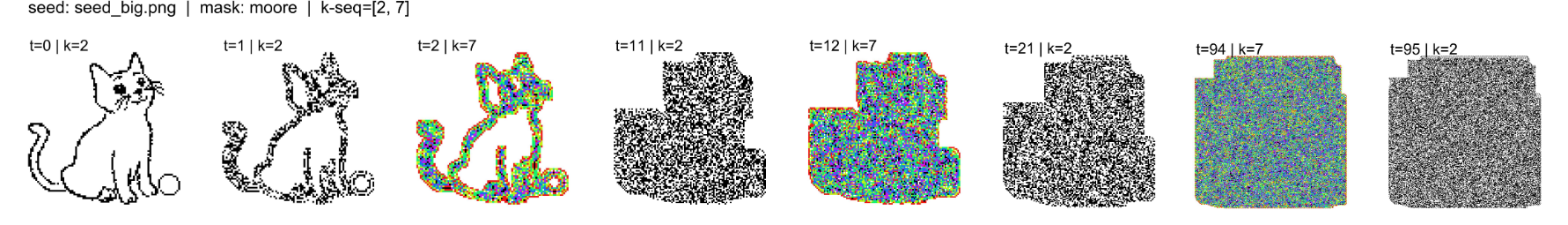}\\
		(b)\includegraphics[width=0.82\textwidth]{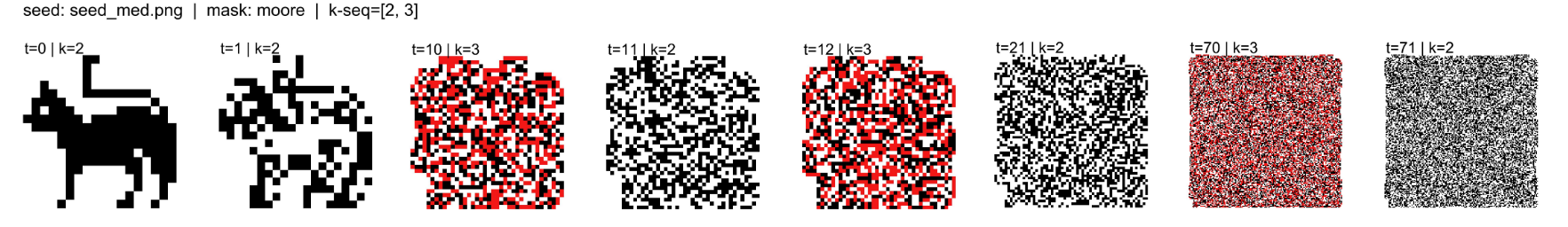}

		\caption{Big and medium seeds under alternating one-to-one 2$k$2$k$... iterations, $k>2$ odd, (a) medium and (b) big seeds on Moore mask, $k=3, 7$, resp. resulting in non-periodic figures becoming quickly chaotic - appears quicker in the case of big seeds.}
		\label{2k2k_big}
	\end{figure}
	
			\begin{figure}[h!]
		\centering
		(a)\includegraphics[width=0.82\textwidth]{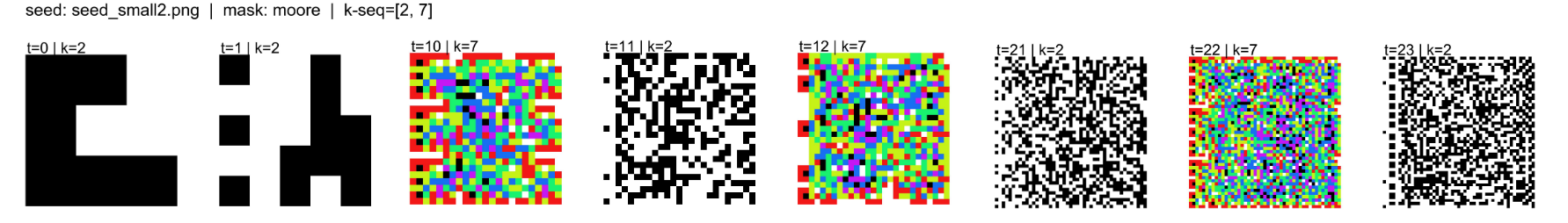}\\
		(a)\includegraphics[width=0.82\textwidth]{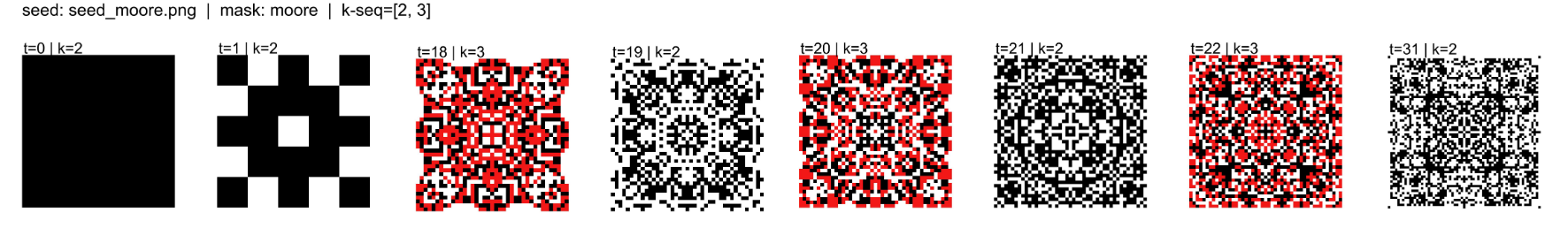}\\
		\caption{Small seeds under alternating one-to-one 2$k$2$k$... sequences, $k>2$ odd, (a) non-symmetric seed and/or mask resultuing in chaotic figures (here $k=7$, Moore mask), (b) symmetric seed and mask resulting in non-periodic, double symmetric figures of relatively stable density ornamental carpet-like figures.}
		\label{2k2k_small}
	\end{figure}
	
Fig. \ref{densities2} is showing typical density pattern of figures obtained via $2k2k\ldots$ sequences, $k\geq 2$, from  small double symmetric seeds and masks, for $k$ even and odd.

	\begin{figure}[h!]
	\centering
	\includegraphics[width=0.82\textwidth]{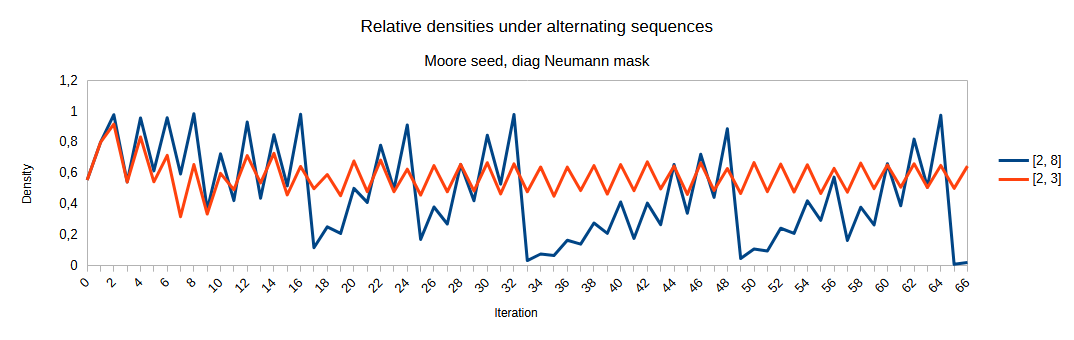}
	\caption{Density patterns of figures obtained under alternating one-to-one 2$k$2$k$... $k\geq 2$, from  small double symmetric seeds and masks, for $k$ even and odd.}
	\label{densities2}
\end{figure}
	
\paragraph{Mini-summary.}
The alternating sequence $[2k]$ with even $k$ preserves many properties 
of the constant binary case: strict periodicity (small periods between 4 and 16, 
depending on seed size) and characteristic density fluctuations. 

For odd $k$, the alternation already destroys binary periodicity. 
No exact seed returns are observed within feasible times. 
Large or medium seeds quickly lead to chaotic figures. 
Small seeds combined with non-symmetric masks behave similarly. 
Only in the case of small, double symmetric seeds and masks do we obtain 
figures with robust symmetry and relatively stable density, though still 
non-periodic. 

Thus, alternating $[2k]$ sequences interpolate between the trivial 
predictability of the binary case and the onset of aperiodic, 
carpet-like complexity.

In next subsection we shall classify and compare different classes of small seeds figures.


\subsection{Special figures: Carpets, Quasi–carpets, Rugs, and Sierpiński–like patterns\label{inherit}}

Building on the replication phenomena described in the previous sections,
we now turn to the classification of long-lived structures. 
These include carpets, quasi–carpets, rugs, and Sierpiński–like patterns, 
which emerge systematically under alternating sequences.

\paragraph{Definitions.}  
We distinguish three main classes:  
\begin{itemize}
	\item \textbf{Rugs}: figures without full double symmetry. Subclasses include:  
	\begin{itemize}
		\item \emph{Chaotic rugs} – irregular, no clear structure, density roughly constant;  
		\item \emph{Disappearing rugs} – self-similar fragments but periodically vanishes; Sierpiński–like rugs belong here (see Fig. \ref{2k2k_small2}(a));  
		\item \emph{Solid rugs} – at least one axis of symmetry, density does not vanish.  
	\end{itemize}
	\item \textbf{Quasi–carpets}: two axes of symmetry, but density fluctuates or decays periodically.  
	\item \textbf{Carpets}: two perpendicular axes of symmetry ("double symmetry"), density remains bounded away from zero and persists over long horizons.  
\end{itemize}

\paragraph{Shape inheritance.}   
As shown in Subsection 2.1, the overall outline of carpets and rugs 
is determined by the neighborhood mask, e.g. squares for Neumann/Moore, 
diamonds for diagonal Neumann, and so on (Table~\ref{tab:shape_inheritance}).

\paragraph{Symmetry inheritance.}  
The overall symmetry of the emerging figure depends jointly on the 
seed and the mask. Table~\ref{tab:symmetry} summarizes the observed outcomes.  

\begin{table}[h!]
	\centering
	\caption{Symmetry inheritance in alternating sequences $[2,k]$. 
		Rows: type of seed; columns: type of mask. The entries describe the 
		expected large-scale outcome.}
	\label{tab:symmetry}
	\renewcommand{\arraystretch}{1.3}
	\begin{tabular}{|p{2.8cm}|p{3.3cm}|p{3.3cm}|p{3.3cm}|}
		\hline
		\textbf{Seed $\backslash$ Mask} & 
		\textbf{Double-symmetric (2 axes)} & 
		\textbf{Single-symmetric (1 axis)} & 
		\textbf{Asymmetric (no symmetry)} \\
		\hline
		\textbf{Double-symmetric (2 axes)} 
		& Double-symmetric  
		& Single-axis symmetry survives 
		& Chaotic; no carpet \\
		\hline
		\textbf{Single-symmetric (1 axis)} 
		& Single-axis symmetry persists
		& Single-axis symmetry persists if axes concordant  
		& Chaotic; no carpet \\
		\hline
		\textbf{Asymmetric (no symmetry)} 
		& Single-axis symmetry may appear; no carpet  
		& One-axis reflection possible; no carpet  
		& Fully chaotic; no carpet \\
		\hline
	\end{tabular}
\end{table}

In short, Only the double $\times$ double case may lead to quasi–carpets or carpets; all other combinations degrade.

\paragraph{Examples.}  
\begin{itemize}
	\item \textbf{Rugs.} Fig.~\ref{fig:rug} shows a \emph{symmetric rug} with pentagonal outline (rocket mask).  
	Disappearing rugs arise for example in sequences $[2422]$ with triangular masks, often producing Sierpiński–like forms.  
	Chaotic rugs occur for non-symmetric seeds with conflicting symmetry axes (rocket–tannenbaum, roof–kite).  
	\item \textbf{Quasi–carpets.} We have already seen quasi- carpets appearing in $2k22...$ sequences for $k$ even, when both seed and mask are double symmetric. Figs.~ \ref{fig:quasi_hex} display high–symmetry figures with oscillating densities.  
	\item \textbf{Carpets.} Genuine carpets occur for sequences, e.g. $[2322]$, $[2522]$, $[2722]$ under Moore, Neumann, or diag–Neumann masks (Fig.~\ref{2k22carpet}).  
\end{itemize}
\begin{figure}[h!]
	\centering
	\includegraphics[width=0.77\textwidth]{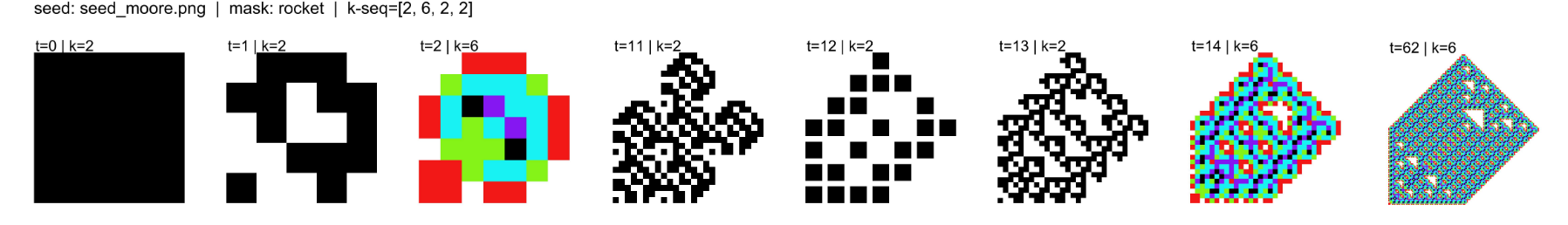}
	\caption{Example of a rug. 
		The mask is \emph{rocket}, seed: Moore. 
		The resulting figure forms a dense pentagonal tiling that is ``floor–tile–like'', 
		with high relative density. However, due to the lack of double symmetry it is 
		classified as a rug rather than a carpet.}
	\label{fig:rug}
\end{figure}

\begin{figure}[h!]
	\centering
	(a)\includegraphics[width=0.75\textwidth]{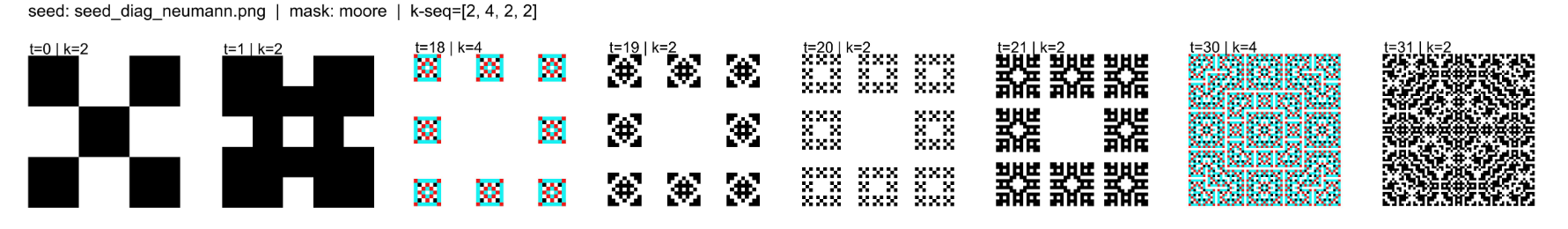}\\
	(b)\includegraphics[width=0.75\textwidth]{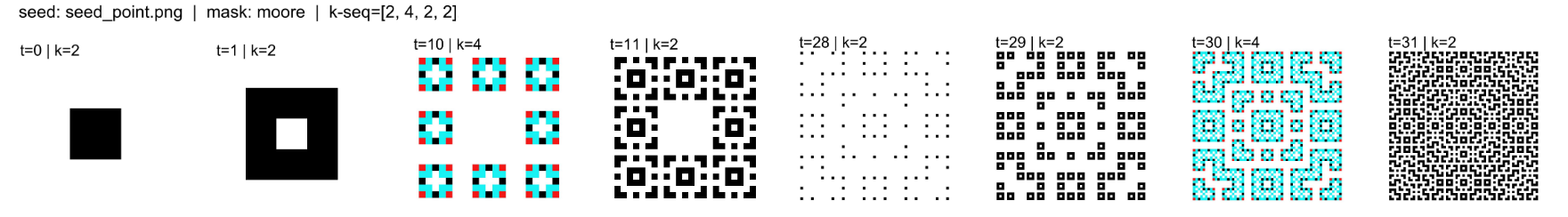}
	\caption{Examples of quasi–carpets on the Moore mask. 
		Figures exhibit two orthogonal symmetry axes but density oscillates 
		and periodically decays, preventing stabilization into true carpets. 
		Seeds: (a) one–point, (b) diag-Neumann, (c) Moore, sequence [2422].}
	\label{fig:quasi_hex}
\end{figure}

\begin{figure}[h!]
	\centering
	(a)\includegraphics[width=0.75\textwidth]{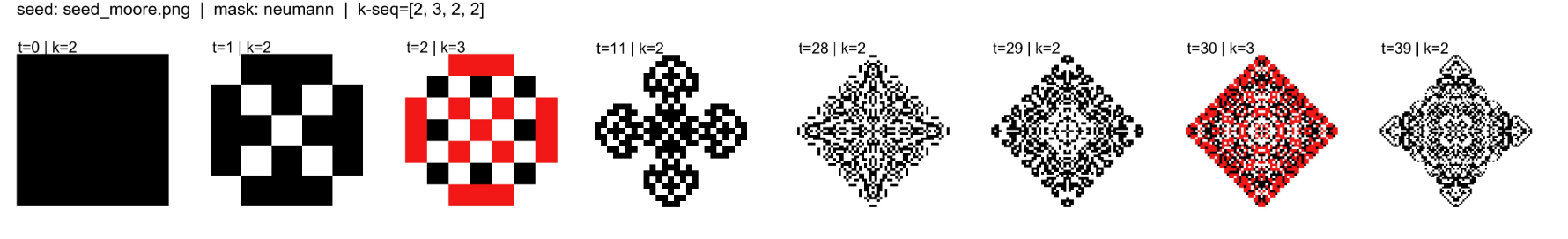}\\
	(b)\includegraphics[width=0.75\textwidth]{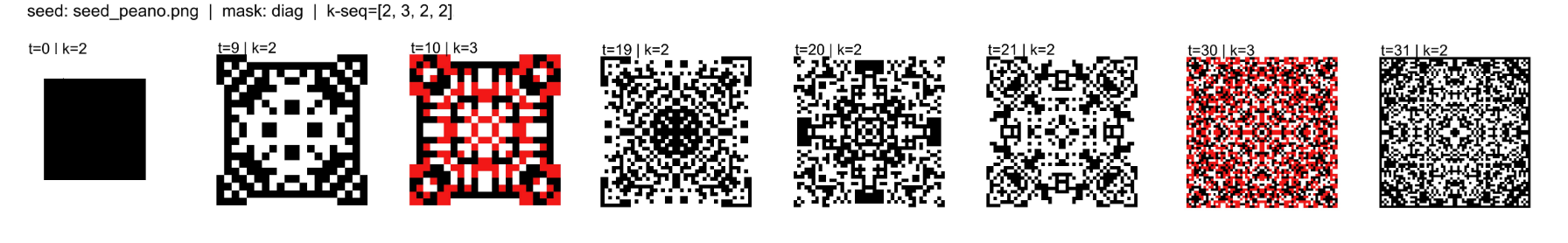}
	\caption{Carpets: Figures obtained via repetitive [2322] sequence, with symmetric seed and mask, resulting in non-periodic figures in the form of sophisticated ornamental carpet-like figures of relatively stable density: 
		(a) point seed, diag Neumann mask,  
		(b) Peano seed, diag Neumann mask.}
	\label{2k22carpet}
\end{figure}

\paragraph{Mini-summary.}
Carpets, quasi–carpets, and rugs provide a natural taxonomy of 
long-lived Laplacian figures. Rugs lack double symmetry and may 
exhibit chaotic or disappearing behavior. Quasi–carpets are 
double-symmetric but unstable in density. True carpets arise only 
for double-symmetric seeds and masks under selected alternating 
sequences (notably $[2k22]$), combining symmetry with sustained density. 
In all cases, the large-scale outline respects the geometry of the mask, 
confirming the robustness of shape inheritance.

\subsection{Regular sequence $2k222\ldots$}
We now investigate binary sequences with a single insertion of another modulus $k$, i.e. $2k222\ldots$. The main questions are whether periodicity survives, how density behaves, and under what conditions carpets or quasi–carpets may arise. As before we shall consider two cases: $k$ even and odd.

It occurs that 
\begin{enumerate}
	\item For big and medium seeds we observe approximate periodicity with shift: 
	the figure at $t=2$ reappears only vaguely at multiples of 32 (small period), 
	while otherwise the evolution becomes chaotic (see Fig.~\ref{fig:density-fingerprint}).
	\item For small seeds the figures resemble the binary case, showing periodic repetition of motifs 
	(small period 8 or 16) and relative density drops. 
	In the case of double–symmetric seeds and masks, quasi–carpets appear.
\end{enumerate}

The relative density pattern occuring for all $k$, seeds and masks is given in Figs \ref{2k2infty}, \ref{densities_infty}. 
For even $k$, density drops occur at multiples of 16, indicating periodicity.
For odd $k$, the same regularity appears but with a phase shift of 2, hence the seed never returns exactly.

\begin{figure}[h!]
	\centering
	(a)\includegraphics[width=0.72\textwidth]{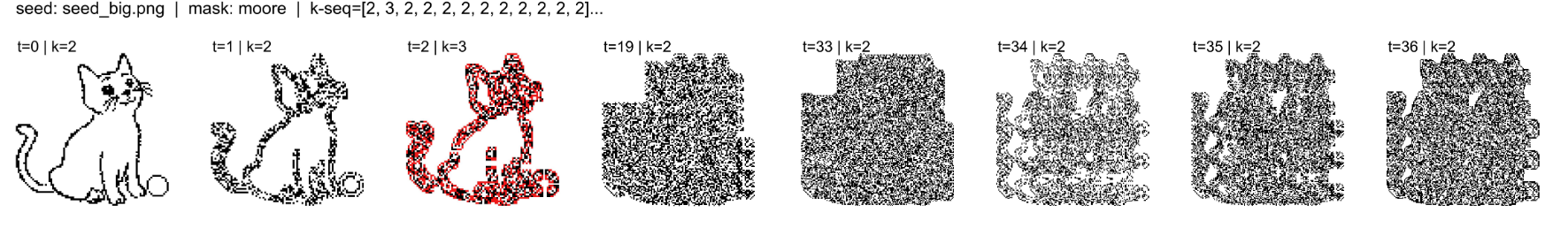}\\
	(c)\includegraphics[width=0.72\textwidth]{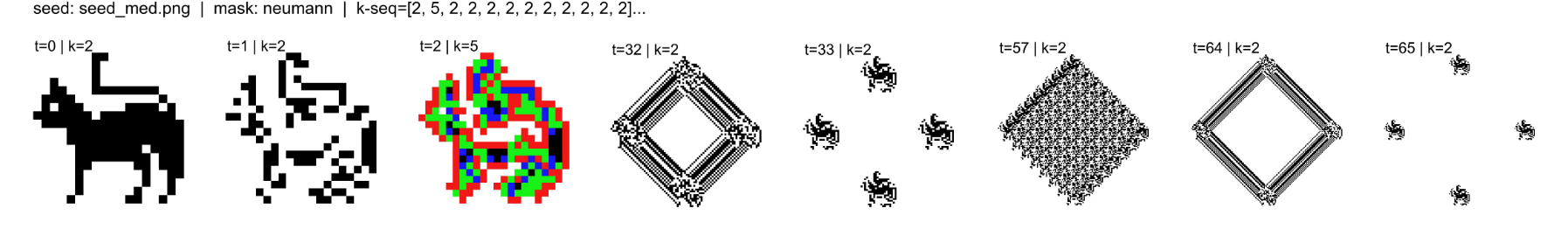}\\
	(b)\includegraphics[width=0.72\textwidth]{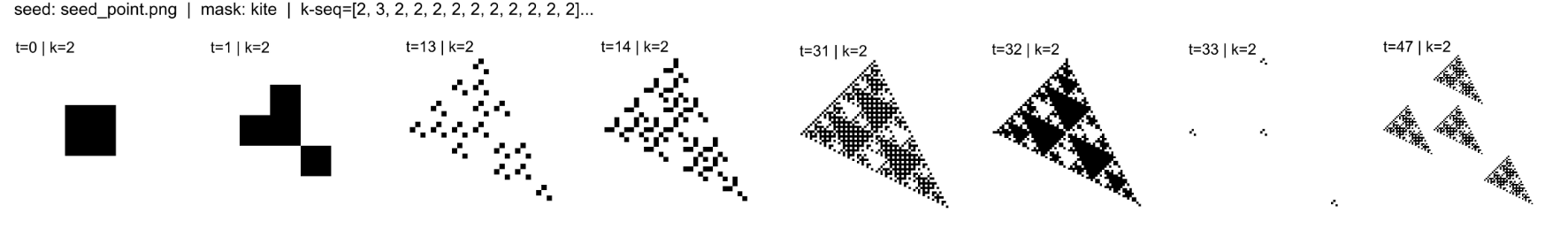}
	\caption{Examples of 2k22... figures, big medium and small seeds. 
		Unstable density is visible.}
	\label{2k2infty}
\end{figure}

\begin{figure}[h!] 
	\centering 
	\includegraphics[width=0.9\textwidth]{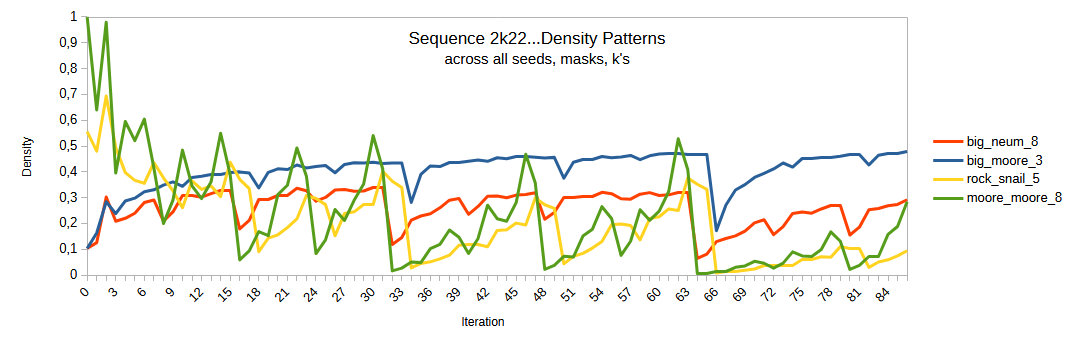} 
	\caption{Density patterns in $2k222\ldots$ sequence characteristic for all $k$'s, seeds and masks. Note the periodicity of 16 with even $k$ sequence, and the phase shift in case of odd $k$'s.} 
	\label{densities_infty} 
\end{figure}

\paragraph{Mini-summary.}
The sequence $[2k22\ldots]$ retains the binary periodicity pattern, 
with period~16. For even $k$, this periodicity leads to regular density drops, 
closely resembling the binary case. For odd $k$, the same structure is shifted 
(by two steps), so exact seed returns never occur. 
Thus, $[2k22\ldots]$ sequences interpolate between the strict periodicity 
of the binary case and the irregular behavior of alternating sequences.

\subsection{Carpets in the \texorpdfstring{$[2,k,2^s]$}{[2,k,2^s]} Family}

We now examine whether inserting a single non–binary step $k$ after the initial $2$,
followed by a binary tail, can prevent the characteristic density drop of the purely
binary sequence and thereby enable the formation of carpets (given suitable seeds and masks).
  
That is, we study sequences of the form
\[
[2,k,\underbrace{2,2,\ldots,2}_{s \text{ times}}],
\qquad s=1,\dots,36,\;\; k\in\{3,5,7,9,11\}.
\]

\paragraph{Method.}  
Only seed–mask configurations showing initial double symmetry were 
considered as candidates.  
For each sequence $[2,k,2^s]$ we simulated up to $t=80$ iterations 
(beyond the first universal decay around $t\approx 32$).  
A configuration was accepted as a carpet if it satisfied simultaneously:
\begin{itemize}
	\item density never dropped below $0.056$ of the bounding box area,
	\item no horizontal or vertical empty stripe wider than $10\%$ of the lattice appeared,
	\item no central hole larger than the same threshold was present.
\end{itemize}

\paragraph{Results.}  
Contrary to the initial expectation, the family $[2,k,2^s]$ does 
indeed contain carpets for suitable values of $k$ and $s$.  
Table~\ref{tab:2k2s} shows how often carpets appeared for different 
choices of $k$. $s=1,2,\ldots,8$, aggregated over all seeds and masks.  
The first ten rows are presented below; the full table (up to $s=36$) 
is given in Appendix~A.

\begin{table}[ht]
	\centering
	\caption{Number of carpet–producing cases in the family $[2,k,2^s]$ 
		for different $k$ and binary tail length $s$ with all small double-symmetric seeds and masks. 
		Explicit notation is used up to $s=3$, longer tails are denoted by powers of~2. Note that even numbers $k$ are lacking and the number $k=3$ the most prominent. Counts aggregated over all tested small double–symmetric seeds and masks; see Appendix A for extended data.}
	\label{tab:2k2s}
	\begin{tabular}{lcccccccc}
		\hline
		$k$ & $[2,k,2]$ & $[2,k,22]$ & $[2,k,222]$ & $[2,k,2^4]$ & $[2,k,2^5]$ & $[2,k,2^6]$ & $[2,k,2^7]$ & $[2,k,2^8]$ \\
		\hline
		3  & 12 &  8 &  9 & 10 &  9 &  8 &  4 &  5 \\
		5  &  9 &  6 &  7 &  3 &  8 &  3 &  2 &  2 \\
		7  &  9 &  7 &  7 &  4 &  8 &  4 &  2 &  3 \\
		9  &  7 &  5 &  5 &  2 &  7 &  1 &  1 &  0 \\
		11 &  7 &  5 &  5 &  2 &  7 &  1 &  1 &  0 \\
		\hline
	\end{tabular}
\end{table}

\paragraph{Discussion.}  

The results show that inserting a single non–binary step $k$ immediately after
the leading $2$ can stabilize carpet–like structures for a surprisingly broad
range of binary tails $2^s$. Without such a modification, the purely binary
sequence undergoes a decisive density collapse by about the 16th iteration. 
By contrast, the $[2,k,2^s]$ family maintains sufficient density and spatial
uniformity well beyond this critical point, in many cases up to $t\approx 35$–37,
thus enabling the emergence of genuine carpets (Table~\ref{tab:2k2s}). 

It is also noteworthy that even values of $k$ never produced carpets,
whereas odd $k$, especially $k=3$, accounted for the majority of positive cases.

\paragraph{Mini-summary.}
Sequences of the form $[2,k,2^s]$ provide the minimal modification of the 
binary case that is capable of sustaining genuine carpets. 
While the purely binary sequence inevitably collapses in density around 
iteration~16, the insertion of a single non–binary step $k$ delays this 
collapse and stabilizes the figure up to $t\approx 35$–37. 
Carpets emerge only for odd $k$, with $k=3$ producing the richest variety 
of cases. Even $k$ never produced carpets under our criteria. 
Thus, $[2,k,2^s]$ sequences mark the transition from the trivial periodic 
behavior of the binary family to long–lived, carpet–like structures.

\subsection{Example: repetitive sequence [2$k$22]}

We next consider repetitive schedules of the form $[2k22]$, 
which represent a minimal extension of the $[2,k,2^s]$ family. 
These sequences offer a natural test case for the stability of 
carpet–like figures under repeated alternation.

Figures obtained via Repetitive sequences [2$k$22] can be seen in Fig. \ref{2k22} and solely obtained by [2322] in Fig. \ref{2322}.
\begin{figure}[h!]
	\centering
	(a)\includegraphics[width=0.8\textwidth]{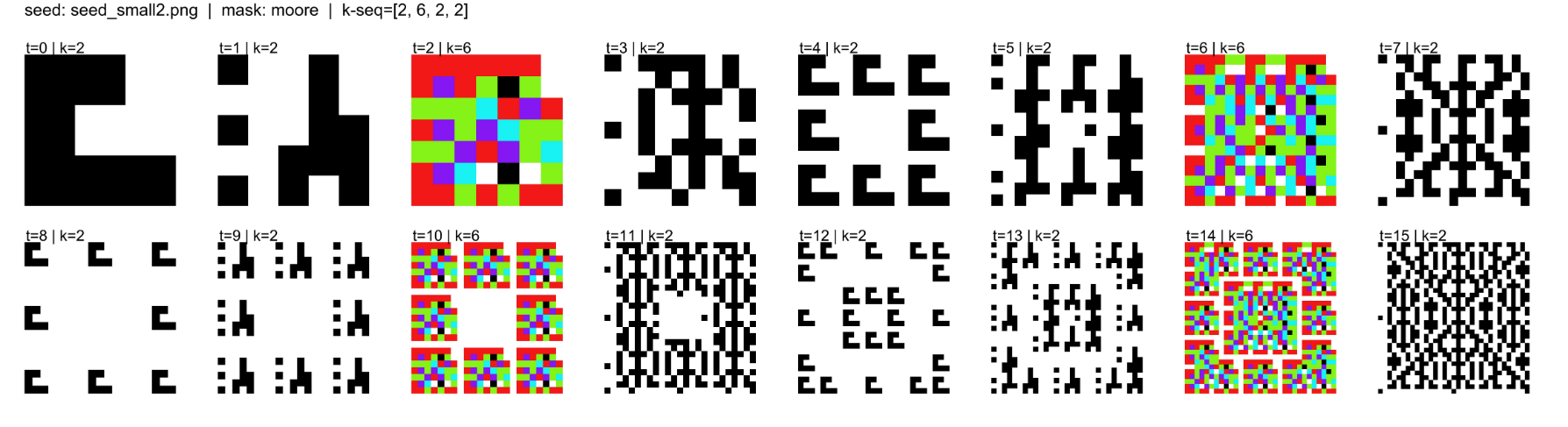}\\
	(b)\includegraphics[width=0.8\textwidth]{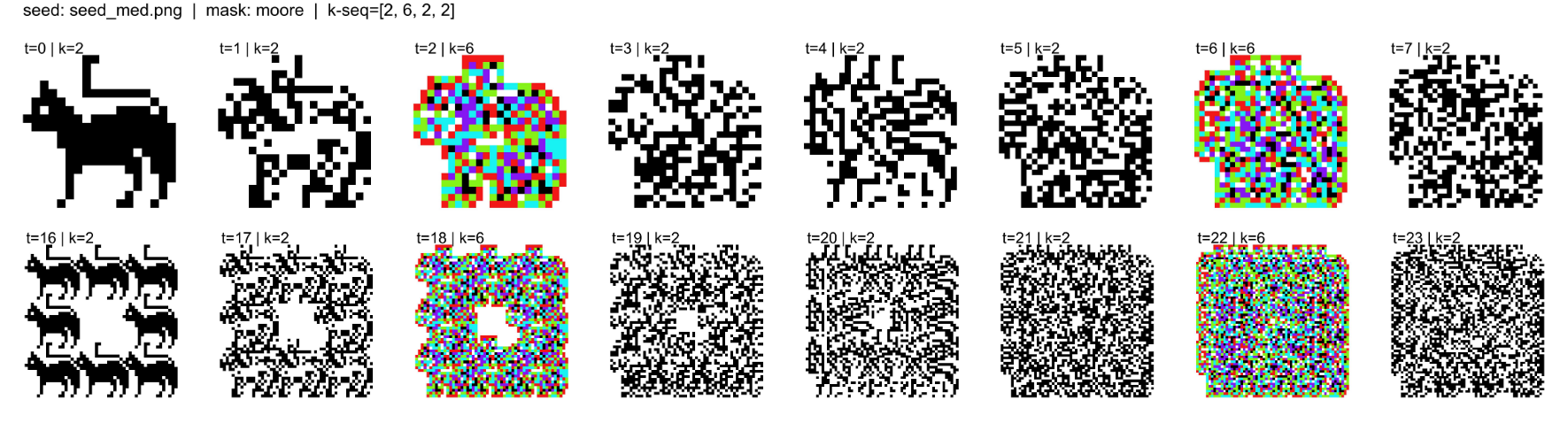}
	\caption{Figures obtained via repetitive [2622] sequence, Moore mask: (a) small seed, period 4 (b) medium period 32.}
	\label{2k22}
\end{figure}
In sequences [$2k22$], $k$ odd, in case of small double-symmetric seeds and masks we encounter ornamental carpet-like figures, see Fig. \ref{2322} (b).

\begin{figure}[h!]
	\centering
	(a)\includegraphics[width=0.8\textwidth]{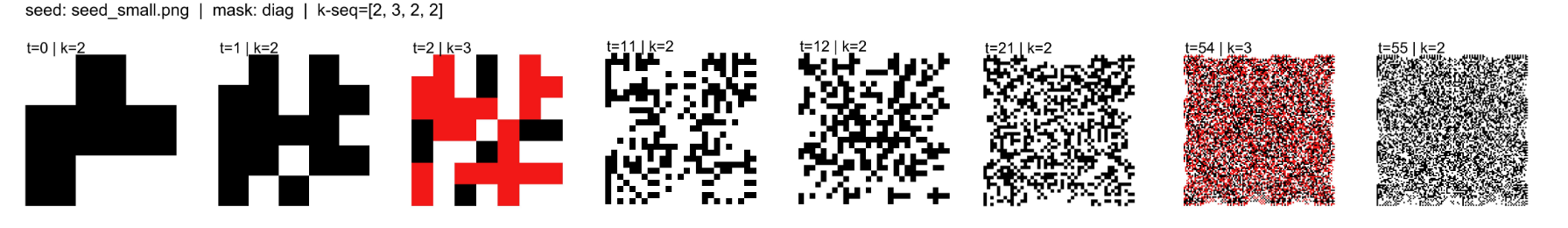}\\
	(b)\includegraphics[width=0.8\textwidth]{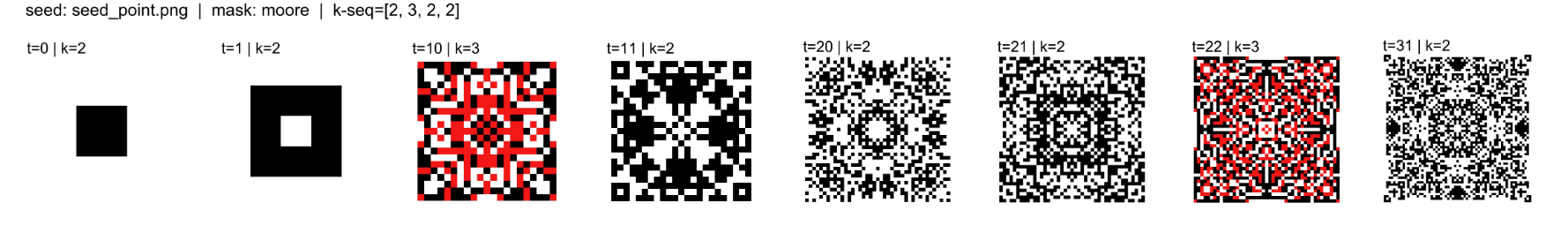}
	(c)\includegraphics[width=0.8\textwidth]{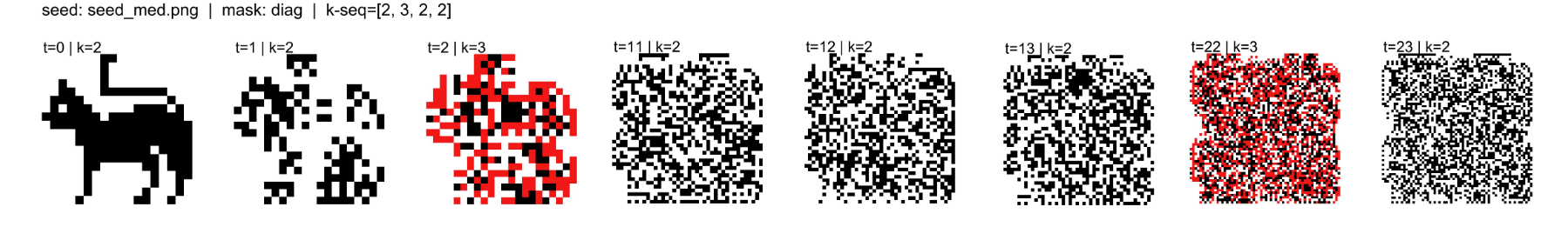}\\
	(d)\includegraphics[width=0.8\textwidth]{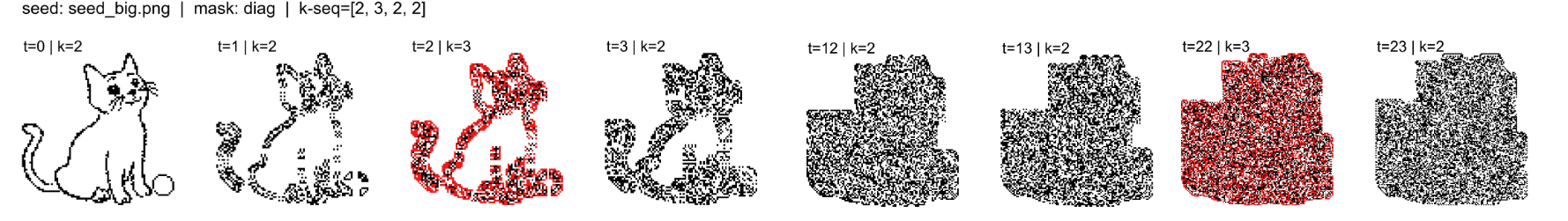}
	\caption{Figures obtained via repetitive [2322] sequence, diag Neumann mask: (a) small non-symmetric seed, chaotic figures (b) symmetric seed and mask: ornamental carpet-like figures, (c) medium and (d) big seeds resulting in non-periodic figures becoming quickly chaotic - appears quicker in the case of big seeds.}
	\label{2322}
\end{figure}

A collection of carpets can also be found in the Appendix.

\paragraph{Comparison: prime vs.\ composite $k$.}
For prime $k$ (e.g.\ $3,5,7$), replication follows a clear $k$-adic law,
with regular grids appearing at multiples of a characteristic base period.
Composite moduli such as $k=6$ exhibit intertwined subperiods inherited
from their prime factors (here $2$ and $3$), which can interfere
constructively or destructively.
This results in irregular or ``beating'' periodicities and makes $k=6$
a borderline case: neither purely even (binary-like) nor purely prime.

This mirrors the dichotomy observed in constant-$k$ sequences: primes support regular replication laws, while composites mix incompatible scales.

\paragraph{Mini-summary.}
For $k=3$, the $[2322]$ schedule produces the richest carpet-like figures: 
high–density, symmetric patterns that never return to the seed, thus realizing 
genuine non-periodic carpets. For other odd $k$, seed returns reappear with 
periods that are typically small multiples of $8$, resembling the binary case. 
Composite values such as $k=6$ yield mixed behavior with irregular or 
``beating'' periodicities, consistent with the interference of their prime factors.

\begin{table}[h!]
	\centering
	\caption{Carpets appearance depending on seeds and masks in [$2k22$] sequences. Note that even numbers do not appear. 
		Sequences with number 3 appear seventeen times, with 5: ten times, with 7: eight times. Counts refer to observed carpets up to $t=80$ iterations; see Appendix~A for full statistics.}
	\label{tab:carpets}
	\begin{tabular}{l|c|c|c}
		\toprule
		Seed $\backslash$ Mask & Moore & Diagonal & Neumann \\
		\midrule
		Moore   & {[2322], [2522]} & {[2322]} & {[2322], [2722]} \\
		Diagonal& {[2322], [2522], [2722]} & {[2322]} & {[2322]} \\
		Neumann & {[2322], [2722]}        & --- & --- \\
		Peano   & {[2322], [2522], [2722],} & {[2322], [2522], [2722]} & --- \\
		Point   & {[2322], [2522], [2722]} & {[2322]} & {[2322]} \\
		\bottomrule
	\end{tabular}
\end{table}

(Color entropy plots are provided in Appendix~\ref{app:entropy}. 
They exhibit only weak signatures of periodicity and do not sharpen the 
contrast between replicative and shifted regimes.)

\paragraph{Transition to Conclusions.}
The exploration of repetitive sequences $[2k22]$ completes our survey of
minimal binary perturbations. Across constant, alternating, and repetitive
families we observed a consistent pattern: prime values of $k$, especially
$k=3$, promote stable carpet–like structures, while composite $k$ introduce
interfering subperiods that disrupt long-term regularity. 
These findings highlight the delicate balance between local update rules
(seed and mask symmetries) and global sequence design in sustaining
non-trivial, high-density figures. 
We now turn to the concluding discussion, where we summarize these results
and outline open problems.

\section*{Summary and Discussion}

In this work we explored the dynamics of binary and mixed sequences of the 
form $[2,k,2^s]$ under a variety of seeds and neighborhood masks. 
The analysis revealed a surprisingly rigid classification:  

\begin{itemize}
	\item Purely binary sequences $[2,2,2,\dots]$ or constant sequences 
	$[k,k,k,\dots]$ always collapse into strict periodicity, with densities 
	oscillating in predictable cycles. Even apparently irregular seeds eventually 
	crystallize into multiplied replicas on a rigid grid.  
	\item Insertion of a single non-binary step $k$ can dramatically alter 
	the behavior. Large and medium seeds, as well as asymmetric masks, 
	quickly evolve into chaotic figures with no carpet-like order.  
	\item If one element (seed or mask) is symmetric while the other is not, 
	a single-axis reflection may persist, but the overall figure remains unstable.  
	\item Only the combination of double-symmetric seeds with double-symmetric 
	masks, together with odd values of $k$, produces high-density 
	quasi–carpets, often persisting far beyond the universal density drop 
	characteristic of the binary case.  
	\item Composite values of $k$ built from two distinct primes 
	(e.g.\ $k=6$ or $k=10$) form a borderline exception: here we observe 
	quasi-periodic fluctuations rather than strict repetition, 
	but no qualitatively new structures arise.
\end{itemize}

This symmetry–based rule provides a unifying heuristic: the interplay of seed, 
mask, and inserted $k$ predicts whether the system dissolves into chaos, 
retains partial symmetry, or yields a genuine quasi–carpet.

\medskip
We classified the observed figures into three broad families:  
\emph{rugs}, \emph{quasi–carpets}, and \emph{carpets}.  
Although derived from a purely mathematical model, this taxonomy 
mirrors broader principles of self-organization.  
Whitesides and Grzybowski \cite{Whitesides2002} 
showed that across all scales, from molecular clusters to macroscopic
granular matter, simple local rules generate emergent global order.  
More recently, Singh \emph{et al.} \cite{Singh2024NonEquilibrium} 
emphasized that persistent ``living matter-like'' architectures 
arise only under non-equilibrium driving, where transient assemblies 
are stabilized by constant fluxes of energy and matter.

In this perspective, our classification aligns naturally:  
rugs correspond to transient high-density assemblies,  
quasi–carpets to symmetric but fluctuating non-equilibrium figures,  
and carpets to stabilized, long-lived architectures.  
In both the mathematical and experimental contexts, stability arises 
only when symmetry and driving are tuned together, whereas mismatched 
conditions lead to transient or chaotic assemblies.

\section*{Conclusion}

Our results suggest that the classification of carpets is not merely a 
computational curiosity, but provides a conceptual bridge between dynamical 
systems, fractal physics, and experimental self-assembly.  
They demonstrate that even in minimalistic lattice models, strict rules 
govern whether complexity survives or collapses into trivial repetition.  

An intriguing open problem remains: are these carpets truly aperiodic, or do 
they in fact possess an extraordinarily long hidden period, analogous to a 
Poincaré recurrence? Addressing this question would clarify whether quasi–carpets 
are genuinely novel structures, or whether they ultimately belong to the same 
periodic universe as their binary ancestors.

Looking ahead, future work could include a systematic classification of masks 
by their ability to sustain carpets, and direct comparisons with experimental 
realizations of self-assembled fractals. Such cross-talk between mathematics 
and chemistry may ultimately clarify how minimal discrete rules can generate 
architectures that resemble those of living matter.

\medskip
In summary, discrete Laplacian dynamics offer a fertile ground where 
mathematics and self-assembly meet. The discovery of carpets, rugs, and 
quasi–carpets shows that even the simplest iterative rules can yield 
architectures of surprising richness. Clarifying whether these structures 
are truly aperiodic, and how they relate to experimental realizations, 
remains an open challenge that connects pure mathematics to the physics 
and chemistry of complex systems.

	\subsection*{Acknowledgments}  
The author thanks anonymous reviewers for their valuable comments.

\subsection*{Declaration of interest statement}  
The author reports there are no competing interests to declare.

\section*{Appendix}

\begin{figure}[h!]
	\centering
	\begin{tikzpicture}
		\begin{axis}[
			width=0.8\linewidth, height=0.36\linewidth,
			xmin=0, xmax=40, ymin=0, ymax=0.9,
			xlabel={$t$}, ylabel={$\rho_t$},
			legend style={at={(0.02,0.98)},anchor=north west, cells={anchor=west}, font=\small},
			tick align=outside, tick pos=left
			]
			
			\addplot+[thick,blue] coordinates {
				(0,0.10) (2,0.28) (4,0.42) (6,0.55) (8,0.12)
				(10,0.30) (12,0.46) (14,0.57) (16,0.11)
				(18,0.29) (20,0.45) (22,0.56) (24,0.12)
				(26,0.28) (28,0.44) (30,0.55) (32,0.11)
				(34,0.30) (36,0.46) (38,0.57) (40,0.12)
			};
			\addlegendentry{$222\ldots$ (periodic, seed return)}
			
			\addplot+[thick,red] coordinates {
				(0,0.46) (2,0.08) (4,0.24) (6,0.39) (8,0.50)
				(10,0.09) (12,0.25) (14,0.41) (16,0.52)
				(18,0.10) (20,0.26) (22,0.40) (24,0.51)
				(26,0.09) (28,0.27) (30,0.41) (32,0.53)
				(34,0.11) (36,0.27) (38,0.42) (40,0.54)
			};
			\addlegendentry{$23222\ldots$ (shifted, phase 2)}
			
			\addplot+[thick,orange] coordinates {
				(0,0.44) (2,0.12) (4,0.29) (6,0.43) (8,0.48)
				(10,0.13) (12,0.30) (14,0.44) (16,0.50)
				(18,0.14) (20,0.31) (22,0.46) (24,0.52)
				(26,0.15) (28,0.32) (30,0.47) (32,0.53)
				(34,0.16) (36,0.33) (38,0.48) (40,0.54)
			};
			\addlegendentry{$25222\ldots$ (shifted, phase 3)}
			
			\addplot+[thick,green!70!black] coordinates {
				(0,0.42) (2,0.14) (4,0.28) (6,0.41) (8,0.47)
				(10,0.15) (12,0.29) (14,0.42) (16,0.49)
				(18,0.16) (20,0.30) (22,0.44) (24,0.50)
				(26,0.17) (28,0.31) (30,0.45) (32,0.51)
				(34,0.18) (36,0.32) (38,0.46) (40,0.52)
			};
			\addlegendentry{$27222\ldots$ (shifted, phase 3)}
			
		\end{axis}
	\end{tikzpicture}
	\caption{Relative density $\rho_t$ for representative iteration sequences. 
		Binary $222\ldots$ shows regular minima at $t\equiv 0 \pmod{8}$ (periodic with seed return). 
		The ternary–triggered $23222\ldots$ exhibits shifted periodicity (phase 2, no seed return). 
		Similarly, $25222\ldots$ and $27222\ldots$ yield shifted cycles (phase 3).}
	\label{fig:density-fingerprint}
\end{figure}
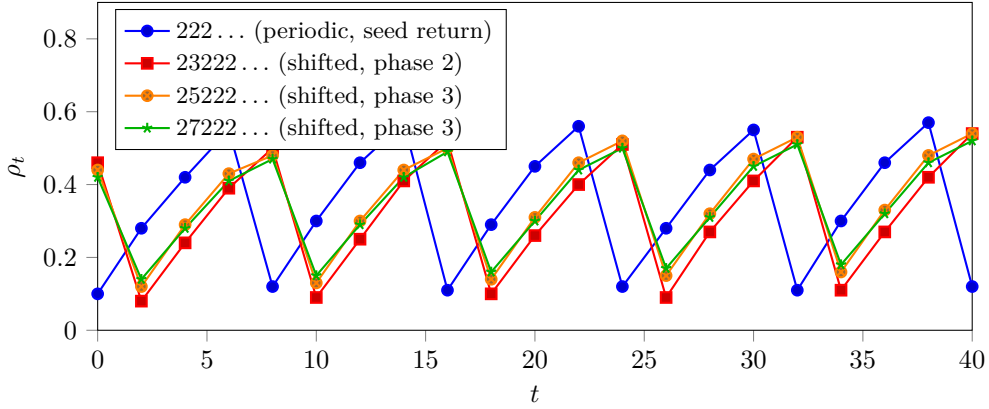

\begin{figure}[h!]
	\centering
	\begin{tikzpicture}
		\begin{axis}[
			width=0.8\linewidth, height=0.36\linewidth,
			xmin=0, xmax=40, ymin=0, ymax=2.0,
			xlabel={$t$}, ylabel={$H_t$},
			legend style={at={(0.02,0.98)},anchor=north west, cells={anchor=west}, font=\small},
			tick align=outside, tick pos=left
			]
			
			\addplot[gray!50, dashed] coordinates {(0,0) (0,2)};
			\addplot[gray!50, dashed] coordinates {(8,0) (8,2)};
			\addplot[gray!50, dashed] coordinates {(16,0) (16,2)};
			\addplot[gray!50, dashed] coordinates {(24,0) (24,2)};
			\addplot[gray!50, dashed] coordinates {(32,0) (32,2)};
			\addplot[gray!50, dashed] coordinates {(40,0) (40,2)};
			
			\addplot+[thick,blue] coordinates {
				(0,0.10) (2,0.35) (4,0.62) (6,0.85) (8,0.12)
				(10,0.40) (12,0.68) (14,0.88) (16,0.14)
				(18,0.39) (20,0.66) (22,0.87) (24,0.12)
				(26,0.38) (28,0.65) (30,0.86) (32,0.13)
				(34,0.40) (36,0.67) (38,0.88) (40,0.15)
			};
			\addlegendentry{$H_t$: $222\ldots$ (seed return, phase 0)}
			
			\addplot+[thick,red] coordinates {
				(0,1.10) (2,0.22) (4,0.55) (6,0.78) (8,1.12)
				(10,0.24) (12,0.57) (14,0.79) (16,1.13)
				(18,0.25) (20,0.58) (22,0.80) (24,1.14)
				(26,0.26) (28,0.59) (30,0.82) (32,1.15)
				(34,0.27) (36,0.60) (38,0.83) (40,1.16)
			};
			\addlegendentry{$H_t$: $23222\ldots$ (shift, phase 2)}
			
			\addplot+[thick,orange] coordinates {
				(0,1.05) (2,0.95) (3,0.35) (4,0.70) (6,0.88) (8,1.06)
				(10,0.96) (11,0.36) (12,0.71) (14,0.89) (16,1.07)
				(18,0.97) (19,0.37) (20,0.72) (22,0.90) (24,1.08)
				(26,0.98) (27,0.38) (28,0.73) (30,0.91) (32,1.09)
				(34,0.99) (35,0.39) (36,0.74) (38,0.92) (40,1.10)
			};
			\addlegendentry{$H_t$: $25222\ldots$ (shift, phase 3)}
			
			\addplot+[thick,green!70!black] coordinates {
				(0,1.00) (2,0.98) (3,0.45) (4,0.72) (6,0.90) (8,1.02)
				(10,1.00) (11,0.46) (12,0.73) (14,0.91) (16,1.03)
				(18,1.01) (19,0.47) (20,0.74) (22,0.92) (24,1.04)
				(26,1.02) (27,0.48) (28,0.75) (30,0.93) (32,1.05)
				(34,1.03) (35,0.49) (36,0.76) (38,0.94) (40,1.06)
			};
			\addlegendentry{$H_t$: $27222\ldots$ (shift, phase 3)}
			
		\end{axis}
	\end{tikzpicture}
	\caption{Color entropy $H_t$ (Appendix).  
		Entropy shows only weak signatures of the phase structure: 
		for $222\ldots$ shallow dips occur at $t\equiv 0 \ (\mathrm{mod}\ 8)$ (seed return), 
		while for shifted sequences $23222\ldots$, $25222\ldots$, $27222\ldots$ 
		the minima are phase–shifted (near $t\equiv 2$ or $3 \ (\mathrm{mod}\ 8)$) and less pronounced. 
		Overall, $H_t$ does not sharpen the contrast beyond relative density.}
	\label{app:entropy}
\end{figure}
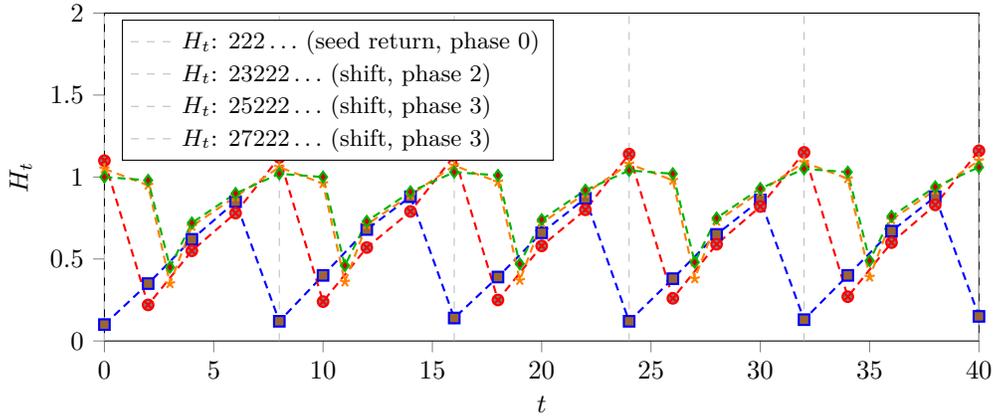

\begin{table}[ht]
	\centering
	\caption{Number of carpets obtained by introducing periodicity to the sequence $2k22...$ with all small double-symmetric seeds and masks. Extended results for $[2,k,2^s]$, showing the number of carpets found for $s=1,\dots ,36$.}
	\label{tab:2k2s_full}
	\begin{tabular}{lccccc}
		\hline
		$s$ & $k=3$ & $k=5$ & $k=7$ & $k=9$ & $k=11$ \\
		\hline
		$2^1$  & 12 &  9 &  9 &  7 &  7 \\
		$2^2$  &  8 &  6 &  7 &  5 &  5 \\
		$2^3$  &  9 &  7 &  7 &  5 &  5 \\
		$2^4$  & 10 &  3 &  4 &  2 &  2 \\
		$2^5$  &  9 &  8 &  8 &  7 &  7 \\
		$2^6$  &  8 &  3 &  4 &  1 &  1 \\
		$2^7$  &  4 &  2 &  2 &  1 &  1 \\
		$2^8$  &  5 &  2 &  3 &  -- &  -- \\
		$2^9$  &  5 &  2 &  3 &  1 &  1 \\
		$2^{10}$ & 8 &  4 &  3 &  1 &  1 \\
		$2^{11}$ & 6 &  3 &  4 &  2 &  2 \\
		$2^{12}$ & 4 &  1 &  1 &  -- &  -- \\
		$2^{13}$ & 6 &  2 &  3 &  -- &  -- \\
		$2^{14}$ & 4 &  3 &  3 &  2 &  2 \\
		$2^{15}$ & 4 &  3 &  4 &  2 &  2 \\
		$2^{16}$ & 2 &  2 &  1 &  -- &  -- \\
		$2^{17}$ & 3 &  2 &  -- &  -- &  -- \\
		$2^{18}$ & 4 &  -- &  1 &  -- &  -- \\
		$2^{19}$ & 4 &  2 &  2 &  1 &  1 \\
		$2^{20}$ & 4 &  2 &  1 &  1 &  1 \\
		$2^{21}$ & 3 &  1 &  1 &  -- &  -- \\
		$2^{22}$ & 2 &  -- &  -- &  -- &  -- \\
		$2^{23}$ & 3 &  -- &  -- &  -- &  -- \\
		$2^{24}$ & 3 &  -- &  -- &  -- &  -- \\
		$2^{25}$ & 3 &  -- &  -- &  -- &  -- \\
		$2^{26}$ & 3 &  -- &  -- &  -- &  -- \\
		$2^{27}$ & 3 &  -- &  -- &  -- &  -- \\
		$2^{28}$ & 3 &  -- &  -- &  -- &  -- \\
		$2^{29}$ & 3 &  -- &  -- &  -- &  -- \\
		$2^{30}$ & 3 &  -- &  -- &  -- &  -- \\
		$2^{31}$ & 3 &  -- &  -- &  -- &  -- \\
		$2^{32}$ & 3 &  -- &  -- &  -- &  -- \\
		\hline
	\end{tabular}
\end{table}

\begin{figure}[h!] 
	\centering 
	\includegraphics[width=0.97\textwidth]{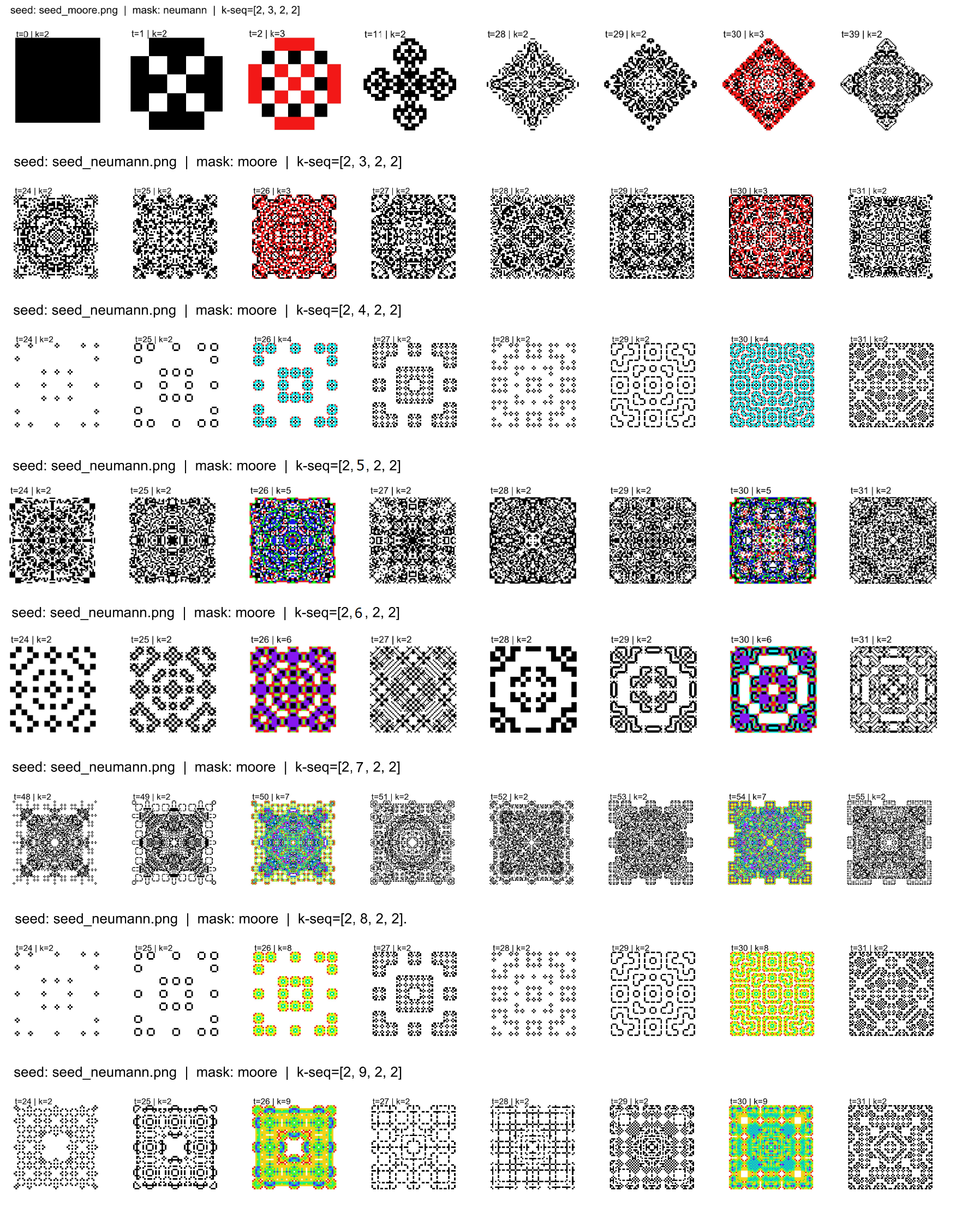} 
	\caption{Carpets and quasi-carpets obtained in [2k22] iterations. 
		For $k=3,5$ we have carpets (first and third row), for $k=4,6,8,9$ quasi–carpets. 
		Seed: Neumann, mask: Moore (except the top: mask diag Neumann).} 
	\label{carpets} 
\end{figure}

\bibliographystyle{amsplain}
\bibliography{bibliography} 

\end{document}